\documentclass[12pt]{article}
\usepackage{lipsum}

\usepackage{amsfonts}
\usepackage{amssymb}
\usepackage{amsthm}
\usepackage{amsmath}
\usepackage{mathrsfs}
\usepackage{xcolor}
\usepackage[latin1]{inputenc}
\usepackage[all]{xy}
\usepackage[left=1.20in,right=1.20in]{geometry}
\definecolor{ao(english)}{rgb}{0.0, 0.5, 0.0}
\definecolor{brickred}{rgb}{0.8, 0.25, 0.33}
\definecolor{burntorange}{rgb}{0.8, 0.33, 0.0}
\definecolor{beaver}{rgb}{0.62, 0.51, 0.44}
\definecolor{brown(traditional)}{rgb}{0.59, 0.29, 0.0}
\definecolor{ao(english)}{rgb}{0.0, 0.5, 0.0}

\def\rouge{\color{red}}

\def\sou{\underline}
\def\sur{\overline}

\newcommand{\lae}{\varepsilon}
\newcommand{\cp}{\mathcal P}

\newcommand{\binf}{\textrm{Inf}}

\newcommand{\Ind}{\mathrm{Ind}}
\newcommand{\Res}{\mathrm{Res}}
\newcommand{\Def}{\mathrm{Def}}
\newcommand{\Inf}{\mathrm{Inf}}
\newcommand{\Iso}{\mathrm{Iso}}
\newcommand{\Hom}{\mathrm{Hom}}
\newcommand{\End}{\mathrm{End}}
\newcommand{\Id}{\mathrm{Id}}
\renewcommand{\Im}{\mathrm{Im}}

\newcommand{\oviz}[1]{\overleftarrow{#1}}

\theoremstyle{plain}
\newtheorem{teo}{Theorem}
\newtheorem*{teo-non}{Theorem}
\newtheorem{prop}[teo]{Proposition}
\newtheorem{coro}[teo]{Corollary}
\newtheorem{lema}[teo]{Lemma}

\theoremstyle{definition}
\newtheorem{defi}[teo]{Definition}
\newtheorem{nota}[teo]{Notation}

\theoremstyle{remark}
\newtheorem{ejem}[teo]{Example}

\newtheorem{rem}[teo]{Remark}
\def\CC{\mathcal{C}}
\def\CD{\mathcal{D}}
\def\CP{\mathcal{P}}
\def\CY{\mathcal{Y}}
\def\Z{\mathbb{Z}}
\def\mpn{\medskip\par\noindent}
\def\mp{\medskip\par}
\newcommand{\AMod}[1]{#1\hbox{-}\mathrm{Mod}}
\author{Serge Bouc and Nadia Romero\\
}
\title{The center of a Green biset functor
}
\date{}

\begin{document}


\maketitle
\begin{center}
\begin{minipage}{12.5cm}
\begin{footnotesize}
{\bf Abstract:} For a Green biset functor $A$, we define the commutant and the center of $A$ and we study some of their properties and their relationship. This leads in particular to the main application of these constructions: the possibility of splitting the category of $A$-modules as a direct product of smaller abelian categories. We give explicit examples of such decompositions for some classical shifted representation functors. These constructions are inspired by similar ones for Mackey functors for a fixed finite group.\vspace{1ex}
\end{footnotesize}
\end{minipage}
\end{center}
\section*{Introduction}
This paper is devoted to the construction of  two analogues of the center of a ring in the realm of Green biset functors, that is ``biset functors with a compatible ring structure''. 
 For a Green biset functor $A$, we present {\em the commutant}  $CA$ of $A$, defined from  a commutation  property, and {\em the center} $ZA$ of $A$, defined from  the structure of the category of $A$-modules. Both $CA$ and $ZA$ are again Green biset functors. These constructions are inspired by similar ones for Mackey functors for a fixed finite group made in Chapter~12 of \cite{bGfun}. \par
The commutant $CA$ is always a Green biset subfunctor of $A$, and we say that $A$ is {\em commutative} if $CA=A$. Most of the classical representation functors are commutative in that sense. One of them plays a fundamental - we should say {\em initial} - role, namely the Burnside biset functor $B$, as biset functors are nothing but {\em modules} over the Burnside functor. An important feature of the category $\AMod{B}$ is its monoidal structure: given two biset functors $M$ and $N$, one can build their tensor product $M\otimes N$, which is again a biset functor. For this tensor product, the category $\AMod{B}$ becomes a symmetric monoidal category, and a Green biset functor $A$ is a {\em monoid object} in $\AMod{B}$. \par
 More generally, for any Green biset functor $A$, we consider the category $\AMod{A}$ of $A$-modules. We will make a heavy use of the equivalence of categories between $\AMod{A}$ and the category of linear representations of the category $\cp_A$ introduced in Chapter 8 of~\cite{biset} (see also Definition~\ref{PA} below), which has finite groups as objects, and in which the set of morphisms from $G$ to $H$ is equal to $A(H\times G)$. 
The category $\cp_B$ associated to the Burnside functor is precisely the {\em biset category} of finite groups. It is a symmetric monoidal category (for the product given by the direct product of groups), and this monoidal structure induces via Day convolution (\cite{day}) the monoidal structure of $\AMod{B}$ mentioned before. \par
A natural question is then to know when the cartesian product of groups endows the category $\cp_A$ with a symmetric monoidal structure, and we show that this is the case precisely when $A$ is commutative. In this case the category $\AMod{A}$ also becomes a symmetric monoidal category.\par

Even though the definition of the center $ZA$ of a Green biset functor $A$ is fairly natural, showing that it is endowed with a Green biset functor structure (even showing that $ZA(G)$ is indeed a set!) is not an easy task, it requires several and sometimes rather nasty computations. On the other hand, one of the rewarding consequences of this laborious process is that we obtain a description of $ZA(G)$ in terms also of a commutation condition, this time on the morphisms of $\cp_A$. Once we have that $ZA$ is indeed a Green biset functor,  we show some nice properties of it, for instance that there is an injective morphism of Green biset functors from $CA$ to $ZA$. This implies in particular that $ZA$ is a $CA$-module. We show also that in case $A$ is commutative, it is a direct summand of $ZA$ as $A$-modules.

In the last section we work within $ZA(1)$, which, as we will see, coincides with the center of the category $A$-Mod.  Any decomposition of the identity element of $ZA(1)$ as a sum of orthogonal idempotents, fulfilling certain finiteness conditions, allows us to decompose $A$-Mod as a direct product of smaller abelian categories. Moreover, since $CA(1)$ is generally easier to compute than $ZA(1)$, we can also use similar decompositions of the identity element of $CA(1)$ instead, thanks to the inclusion $CA\hookrightarrow ZA$. We give then a series of explicit examples. The first one is the Burnside $p$-biset functor $A=RB_p$ over a ring $R$ where the prime $p$ is invertible. In this case, we obtain an infinite series of orthogonal idempotents in $ZA(1)$, and this shows in particular that $ZA$ can be much bigger than $CA$. Next we consider some classical representation functors, shifted by some fixed finite group $L$ via the {\em Yoneda-Dress functor}. In this series of examples, we will see that the smaller abelian categories obtained in the decomposition are also module categories for Green biset functors arising from the functor $A$, the shifting group $L$, and the above-mentioned idempotents.

\section{Preliminaries}
Throughout the paper, we fix a commutative unital ring $R$.
All referred groups will be finite.  The center of a ring $S$ will be denoted by $Z(S)$.
\subsection{Green biset functors}
The biset category over $R$ will be denoted by $R\CC$. Recall that its objects are all finite groups, and that for finite groups $G$ and $H$, the hom-set $\Hom_{R\CC}(G,H)$ is $RB(H,G)=R\otimes_\Z B(H,G)$, where $B(H,G)$ is the Grothendieck group of the category of finite $(H,G)$-bisets. The composition of morphisms in $R\CC$ is induced by $R$-bilinearity from the composition of bisets, which will be denoted by $\circ$. \par
We fix a non-empty class $\mathcal{D}$ of finite groups closed under subquotients and cartesian products, and a set {\bf D} of representatives of isomorphism classes of groups in $\CD$. We denote by $R\mathcal{D}$ the full subcategory of $R\CC$ consisting of groups in $\CD$, so in particular $R\mathcal{D}$ is a {\em replete subcategory} of $R\CC$, in the sense of \cite{biset}, Definition 4.1.7. The category of biset functors, i.e. the category of $R$-linear functors from $R\CC$ to the category $\AMod{R}$ of all $R$-modules, will be denoted by Fun$_R$. The category Fun$_{\mathcal{D},R}$ of {\em $\CD$-biset functors} is the category of $R$-linear functors from $R\mathcal{D}$ to $\AMod{R}$.

A Green $\CD$-biset functor is defined as a monoid in Fun$_{\CD,R}$ (see Definition 8.5.1 in~\cite{biset}). This is equivalent to the following definition:

\begin{defi}
\label{defgreen}
A $\CD$-biset functor $A$ is a Green $\CD$-biset functor if it is equipped with
bilinear products $A(G)\times A(H)\rightarrow A(G\times H)$ denoted by $(a,\, b)\mapsto a\times b$, for groups $G, H$ in $\CD$, and an identity element $\lae_A\in A(1)$, satisfying the following conditions:
\begin{itemize}
 \item[1.] Associativity. Let $G$, $H$ and $K$ be groups in $\CD$. If we consider  the canonical isomorphism from $G\times (H\times K)$ to $(G\times H)\times K$, then for any $a\in A(G)$, $b\in A(H)$ and $c\in A(K)$
\begin{displaymath}
 (a\times b)\times c=A\left(\Iso_{G\times(H\times K)}^{(G\times H)\times K}\right)(a\times (b\times c)).
\end{displaymath}
\item[2.] Identity element. Let $G$ be a group in $\CD$ and consider the canonical isomorphisms $1\times G\rightarrow G$  and $G\times 1 \rightarrow G$. Then for any $a\in A(G)$
\begin{displaymath}
 a=A\left(\Iso_{1\times G}^G\right)(\lae_A\times a)=A\left(\Iso_{G\times 1}^G\right)(a\times \lae_A).
\end{displaymath}
\item[3.] Functoriality. If $\varphi: G\rightarrow G'$ and $\psi: H\rightarrow H'$ are morphisms in $R\CD$, then for any $a\in A(G)$ and $b\in A(H)$
\begin{displaymath}
 A(\varphi \times \psi)(a\times b)=A(\varphi)(a)\times A(\psi)(b).
\end{displaymath}
 \end{itemize}
\end{defi}
The identity element of $A$ will be denoted simply by $\varepsilon$ if there is no risk of confusion.

If $A$ and $C$ are Green $\CD$-biset functors, a morphism of Green $\CD$-biset functors from $A$ to $C$ is a natural transformation $f:A\rightarrow C$ such that $f_{H\times K}(a\times b)=f_H(a)\times f_K(b)$ for any groups $H$ and $K$ in $\CD$ and any $a\in A(H)$, $b\in A(K)$, and such that $f_1(\lae_A)=\lae_C$. We will denote by Green$_{\CD,R}$ the category of Green $\CD$-biset functors with morphisms given in this way.

There is an equivalent way of defining a Green biset functor, as we see in the {next} lemma.

\begin{defi}
\label{defvieja}
A $\CD$-biset functor $A$ is a Green $\CD$-biset functor provided that for each group $H$ in $\CD$, the $R$-module $A(H)$ is an $R$-algebra with unity that satisfies the following. If $K$ and $G$ are groups in $\CD$ and $K\rightarrow G$ is a group homomorphism, then:
\begin{itemize}
\item[1.]  For the $(K,\, G)$-biset $G$, which we denote by $G_r$, the morphism $A(G_r)$ is a ring homomorphism.
\item[2.] For the $(G,\, K)$-biset $G$, denoted by $G_l$, the morphism $A(G_l)$ satisfies the Frobe\-nius identities for all $b\in A(G)$ and $a\in A(K)$,
\begin{displaymath}
A(G_l)(a)\cdot b=A(G_l)\big(a\cdot A(G_r)(b)\big)
\end{displaymath}
\begin{displaymath}
b\cdot A(G_l)(a)=A(G_l)\big(A(G_r)(b)\cdot a\big),
\end{displaymath}
where $\cdot$ denotes the ring product on $A(G)$, resp. $A(K)$.
\end{itemize}
\end{defi}

\begin{lema}[Lema 4.2.3 in \cite{tesis}]
\label{defeq}
The two previous definitions are equivalent. Starting by Definition \ref{defgreen}, the ring structure of $A(H)$ is given by
\begin{displaymath}
a\cdot b=A\left(\Iso_{\Delta(H)}^H\circ\Res^{H\times H}_{\Delta(H)}\right)(a\times b),
\end{displaymath}
for $a$ and $b$ in $A(H)$, with the unity given by $A(\Inf^{\,H}_1)(\varepsilon)$. Conversely, starting by Definition \ref{defvieja}, the product of $A(G)\times A(H)\rightarrow A(G\times H)$ is given by 
\begin{displaymath}
a\times b=A(\Inf_G^{\,\, G\times H})(a)\cdot A(\Inf_H^{\,\, G\times H})(b)
\end{displaymath}
for $a\in A(G)$ and $b\in A(H)$, with the identity element given by the unity of $A(1)$.
\end{lema}
In what follows, the ring structure on $A(G)$ will be understood as $\big(A(G),\cdot\big)$.\par
Observe that in the case of $A(1)$, the product $\times :A(1)\times A(1)\rightarrow A(1)$ coincides with the ring product $\cdot :A(1)\times A(1)\rightarrow A(1)$, up to identification of $1\times 1$ with $1$, and the unity coincides with the identity element. 

\begin{rem}
A morphism of Green $\CD$-biset functors $f:A\rightarrow C$ induces, in each component~$G$, a unital ring homomorphism $f_G:A(G)\rightarrow C(G)$. Conversely, a morphism of biset functors $f:A\rightarrow C$ such that $f_G$ is a unital ring homomorphism for every $G$ in~$\mathcal{D}$, is a morphism of Green $\CD$-biset functors. \end{rem}

\begin{ejem}
\label{ejemplos}
Classical examples of Green biset functors are the following:
\begin{itemize}
\item[$\bullet$] The Burnside functor $B$. The Burnside group of a finite group $G$ is known to define a biset functor. The cross product of sets defines the bilinear products $B(G)\times B(H)\rightarrow B(G\times H)$ that make $B$ a Green biset functor. The functor $B$ can also be considered with coefficients in $R$, and denoted by $RB=R\otimes_{\mathbb{Z}}B(\,\_\,)$. It is shown in Proposition 8.6.1  of~\cite{biset} that $RB$ is an initial object in Green$_{\CD,R}$. More precisely, for a Green $\CD$-biset functor~$A$, the unique morphism of Green functors $\upsilon_A:RB\to A$ is defined at $G\in\CD$ as the linear map $\upsilon_{A, G}$ sending a $G$-set~$X$ to $A({_GX_1})(\varepsilon_A)$, where ${_GX_1}$ is the set $X$ viewed as a $(G,1)$-biset.
\item[$\bullet$] The functor of $\mathbb{K}$-linear representations, $R_{\mathbb{K}}$, where $\mathbb{K}$ is a field of characteristic~0. That is, the functor which sends a finite group $G$ to the Grothendieck group $R_\mathbb{K}(G)$ of the category of finitely generated $\mathbb{K}G$-modules. Also known to be a biset functor, it has a Green biset functor structure given by the tensor product over $\mathbb{K}$. We will consider the scalar extension $\mathbb{F}R_{\mathbb{K}}=\mathbb{F}\otimes_{\mathbb{Z}}R_{\mathbb{K}}(\,\_\,)$, where $\mathbb{F}$ is a field of characteristic 0.
\item[$\bullet$] The functor of $p$-permutation representations $pp_k$, for $k$ an algebraically closed field of positive characteristic $p$. 
This is the functor sending a finite group $G$ to  the Grothendieck group $pp_k(G)$ of the category of finitely generated $p$-permutation $kG$-modules (also known as trivial source modules), for relations given by direct sum decompositions. The biset functor $pp_k$ is a Green biset functor with products given by the tensor product over the field $k$. When considering coefficients for this functor, we will assume that $\mathbb{F}$ is a field of characteristic 0 containing all the $p'$-roots of unity, and we write $\mathbb{F}pp_k=\mathbb{F}\otimes_{\mathbb{Z}}pp_k(\,\_\,)$. 
\end{itemize}
\end{ejem}
In Section~\ref{some examples} we will focus on the above examples only, but there are many other important examples of Green biset functors, e.g. the monomial Burnside functor - also called the fibred Burnside functor -, which gives rise to fibred biset functors (see \cite{barker-fibred}, \cite{mine2}, \cite{boltje-coskun}), or the slice Burnside functor (see \cite{slisec}, \cite{tounkara}, \cite{tounkara-cras}).\medskip\par
When $p$ is a prime number, and $\CD$ is the full subcategory of $\CC$ consisting of finite $p$-groups, the $\CD$-biset functors are simply called {\em $p$-biset functors}, and their category is denoted by ${\rm Fun}_{p,R}$. Similarly, the Green $\CD$-biset functors will be called {\em Green $p$-biset functors}, and their category will be denoted by ${\rm Green}_{p,R}$.\mpn

An important element in what follows will be the Yoneda-Dress construction. We recall some of the basic results about it, more details can be found in Section 8.2 of~\cite{biset}.
If $G$ is a fixed group in $\CD$ and $F$ is a $\CD$-biset functor, then the Yoneda-Dress construction of $F$ at $G$ is the $\CD$-biset functor $F_G$ that sends each group $K$ in $\CD$ to $F(K\times G)$. The morphism $F_G(\varphi): F(H\times G)\to F(K\times G)$  associated to an element $\varphi$ in $RB(K,\, H)$ 	is defined as $F(\varphi \times G)$. In turn $F(\varphi\times G)$ is defined by $R$-bilinearity from the case where $\varphi$ is represented by a $(K,H)$-biset $U$: in this case $\varphi\times G$ denotes the cartesian product $U\times G$, endowed with its obvious $(K\times G,H\times G)$-biset structure. We also call $F_G$ the functor \textit{shifted} by $G$.\par
If $f:F\rightarrow T$ is a morphism of $\CD$-biset functors, then $f_G:F_G\rightarrow T_G$ is defined in its component $K$ as $(f_G)_K=f_{K\times G}$. It is shown in Proposition 8.2.7 of \cite{biset} that this construction is a self-adjoint exact $R$-linear endofunctor of Fun$_{D,R}$. \par When $A$ is a Green $\CD$-biset functor, the particular shifted functor $A_G$ is also  
a Green $\CD$-biset functor (Lemma 4.4 in \cite{mine})  with product given in the following way:
\begin{displaymath}
A_G(H)\times A_G(K)\rightarrow A_G(H\times K)\quad (a,\, b)\mapsto A(\alpha)(a\times b)
\end{displaymath}
where $\alpha$ is the biset $\Iso_D^{H\times K\times G}\Res^{H\times G\times K\times G}_D$ and $D \cong H\times K\times G$ is the subgroup of $H\times G\times K\times G$ consisting of elements of the form $(h,\,g,\,k,\,g)$. Usually, by an abuse of notation, we will denote this biset simply by $\Res^{H\times G\times K\times G}_{H\times K\times \Delta(G)}$. To avoid confusion with the product $\times$ of $A$ we denote the product of $A_G$ by $\times^d$, where the exponent $d$ stands for {\em diagonal}.

\begin{rem}
It is not hard to show that the ring structure of Lemma \ref{defeq} in $A_G(H)$ induced by the product $\times^d$ of $A_G$ coincides  with the ring structure of $A(H\times G)$ induced by the product $\times$ of $A$. So there is no risk of confusion when talking about  \textit{the ring} $A_G(H)$, since the ring structure we are considering is unique. In particular, the isomorphism $A_G(1)\cong A(G)$ is an isomorphism of rings. 
\end{rem}

\subsection{A-modules}
\label{Amod}

\begin{defi} [Definition 8.5.5 in \cite{biset}] 
Given a Green $\CD$-biset functor $A$, a left $A$-module $M$ is defined as a $\CD$-biset functor, together with bilinear products
\begin{displaymath}
\_\,\times \_ :A(G)\times M(H)\longrightarrow M(G\times H)
\end{displaymath} 
for every pair of groups $G$ and $H$ in $\CD$, that satisfy analogous conditions to those of Definition \ref{defgreen}.  {The notion of right $A$-module is defined similarly, from bilinear products $M(G)\times A(H)\longrightarrow M(G\times H)$.}

\end{defi}

 We use the same notation $\times$ for the product of $A$ and the action of $A$ on $A$-modules, as long as there is no risk of confusion.

If $M$ and $N$ are $A$-modules, a {\em morphism of $A$-modules} is defined as a morphism of $\CD$-biset functors $f:M\rightarrow N$ such that $f_{G\times H}(a\times m)=a\times f_H(m)$ for all groups $G$ and $H$ in $\CD$, $a\in A(G)$ and $m\in M(H)$. With these morphisms, the $A$-modules form a category, denoted by $\AMod{A}$. The category  $\AMod{A}$ {is an abelian subcategory of Fun$_{\CD,R}$}. Actually, the direct sum of biset functors is as well the direct sum of $A$-modules. Also, the kernel, the image and the cokernel of a morphism of $A$-modules are $A$-modules. Basic results {on} modules over a ring can be stated for $A$-modules.\par
In particular, a left (resp. right) ideal of a Green $\CD$-biset functor $A$ is an $A$-submodule of the left (resp. right) $A$-module $A$. A two sided ideal of $A$ is a left ideal which is also a right ideal.

\begin{ejem} If $A$ is the Burnside functor $RB$, then an $A$-module is nothing but a biset functor with values in $\AMod{R}$.
\end{ejem}
From Proposition 8.6.1 of \cite{biset}, or Proposition 2.11 of \cite{mine}, an equivalent way of defining an $A$-module is as an $R$-linear functor from the category $\cp_A$ to $\AMod{R}$, the category $\cp_A$ being  defined next.
\begin{defi} \label{PA}Let $A$ be a Green $\CD$-biset functor over $R$. The category $\cp_A$ is defined in the following way:
\begin{itemize}
 \item The objects of $\cp_A$ are all finite groups {in $\CD$}.
 \item If $G$ and $H$ are groups {in $\CD$}, then ${\Hom}_{\cp_A}(H,\, G)=A(G\times H)$.
 \item Let $H,\, G$ and $K$ be groups {in $\CD$}. The composition of $\beta\in A(H\times G)$ and $\alpha\in A(G\times K)$ in $\cp_A$ is the following:
\begin{displaymath}
\beta \circ \alpha = A\left(\Def^{\,H\times\Delta(G)\times K}_{H\times K}\circ\Res^{H\times G\times G\times K}_{H\times\Delta(G)\times K}\right)(\beta\times\alpha).
\end{displaymath}
\item For a group $G$ in $\CD$, the identity morphism $\varepsilon_G$ of $G$ in $\cp_A$ is $A(\textrm{Ind}_{\Delta(G)}^{G\times G}\circ \textrm{Inf}_1^{\,\Delta (G)})(\varepsilon)$.
\end{itemize}
\end{defi}

Observe that the biset $\Def^{\,H\times\Delta(G)\times K}_{H\times K}\circ\Res^{H\times G\times G\times K}_{H\times\Delta(G)\times K}$
can also be written as 
\begin{displaymath}
H\times\left(\Def^{\Delta(G)}_{1}\circ\Res^{G\times G}_{\Delta(G)}\right)\times K.
\end{displaymath}
Another way of denoting the $(1,\,G \times G)$-biset $\Def^{\Delta(G)}_{1}\circ\Res^{G\times G}_{\Delta(G)}$
is as $\oviz{G}$. In some cases it will be more convenient to use this notation.\medskip\par
 
The category $\cp_A$ is essentially small, as it has a skeleton consisting of our chosen set {\bf D} of representatives of isomorphism classes of groups in $\CD$. Hence, the category Fun$_R(\cp_A,\AMod{R})$ of $R$-linear functors is an abelian category. The above-mentioned equivalence of categories between $\AMod{A}$ and Fun$_R(\cp_A,\AMod{R})$ is built as follows:
\begin{itemize}
\item If $M$ is an $A$-module, let $\widetilde{M}\in{\rm Fun}_R(\cp_A,\AMod{R})$ be the functor defined by:
\begin{enumerate}
\item For $G\in \CD$, we have $\widetilde{M}(G)=M(G)$.
\item For $G,H\in\CD$ and a morphism $\alpha\in A(H\times G)$ from $G$ to $H$ in $\cp_A$, the map $\widetilde{\alpha}:\widetilde{M}(G)\to\widetilde{M}(H)$ is the map sending 
$$m\in M(G)\mapsto M(H\times\oviz{G})(\alpha\times m).$$
\end{enumerate}
\item Conversely if $F\in{\rm Fun}_R(\cp_A,\AMod{R})$, let $\widehat{F}$ be the $A$-module defined by:
\begin{enumerate}
\item If $G\in \CD$, then $\widehat{F}(G)=F(G)$.
\item For $G,H\in\CD$,  $a\in A(G)$ and $m\in F(H)$, set
$$a\times m=F\Big(A\big(\Ind_{G\times\Delta(H)}^{G\times H\times H}\Inf_{G}^{G\times H}\big)(a)\Big)(m)\in F(G\times H),$$
where $A\big(\Ind_{G\times\Delta(H)}^{G\times H\times H}\Inf_{G}^{G\times H}\big)(a)\in A(G\times H\times H)$ is viewed as a morphism from $H$ to $G\times H$ in the category $\cp_A$.
\end{enumerate}
\end{itemize}
Then $M\mapsto\widetilde{M}$ and $F\mapsto\widehat{F}$ are well defined equivalences of categories between $\AMod{A}$ and ${\rm Fun}_R(\cp_A,\AMod{R})$, inverse to each other.\medskip\par

Finally, we extend to $A$-modules our previous definition of the Yoneda-Dress construction.
\begin{defi}
Let $A$ be a Green $\CD$-biset functor. For $L\in \CD$, consider the assignment $\rho_L={-}\times L$ defined for objects $G,H$ of $\cp_A$ and morphisms $\alpha\in \Hom_{\cp_A}(G,H)=A(H\times G)$ by:
$$\left\{\begin{array}{rcl}
\rho_L(G)&=&G\times L\\
\rho_L(\alpha)&=&\alpha\times L:=\Iso_{H\times G\times L\times L}^{H\times L\times G\times L}\big(\alpha\times \upsilon_{A,L\times L}(L)\big)
\end{array}\right.$$
where $\upsilon_{A,L\times L}(L)$ is the image in $A(L\times L)$ of the identity $(L,L)$-biset $L$ under the canonical morphism $\upsilon_{A,L\times L}$, and the isomorphism $H\times G\times L\times L\to H\times L\times G\times L$ maps $(h,g,l_1,l_2)$ to $(h,l_1,g,l_2)$. 
\end{defi}

A straightforward computation shows that 
$$\rho_L(\alpha)=A\big(\Ind_{H\times G\times L}^{H\times L\times G\times L}\Inf_{H\times G}^{H\times G\times L}\big)(\alpha),$$
and this form may be more convenient for calculations. Here $H\times G\times L$ embeds in $H\times L\times G\times L$ via the map $(h,g,l)\mapsto (h,l,g,l)$, and maps surjectively onto $H\times G$ via $(h,g,l)\mapsto (h,g)$.\par
It is easy to check that $\rho_L$ is in fact an endofunctor of $\cp_A$, called the {\em (right) $L$-shift}. It induces by precomposition an endofunctor of the category ${\rm Fun}_{R}(\cp_A,\AMod{R})$, that is, up to the above equivalence of categories, an endofunctor of the category $\AMod{A}$, which can be described as follows. It maps an $A$-module $M$ to the shifted $\CD$-biset functor $M_L$, endowed with the following product: for $G,H\in \CD$, $\alpha\in A(H)$ and $m\in M_L(G)=M(G\times L)$, the element $\alpha\times m$ of $M_L(H\times G)=M(H\times G\times L)$ is simply the element $\alpha\times m$ obtained from the $A$-module structure of~$M$.\par
This endofunctor $M\mapsto M_L$ of the category $\AMod{A}$ will be denoted by $\Id_L$. It is the Yoneda-Dress construction for $A$-modules. 
\begin{rem} \label{bifunctor} For $L\in\CD$, there is another obvious endofunctor $\lambda_L=L\times{-}$ of $\cp_A$ defined for objects $G,H$ of $\cp_A$ and morphisms $\alpha\in \Hom_{\cp_A}(G,H)=A(H\times G)$ by
$$\left\{\begin{array}{rcl}
\lambda_L(G)&=&L\times G\\
\lambda_L(\alpha)&=&L\times \alpha:=\Iso_{L\times L\times H\times G}^{L\times H\times L\times G}\big(\upsilon_{A,L\times L}(L)\times \alpha\big)
\end{array}\right.$$
where  the isomorphism $L\times L\times H\times G\to L\times H\times L\times G$ maps $(l_1,l_2,h,g)$ to $(l_1,h,l_2,g)$. As before, it is easy to see that $L\times\alpha=A\big(\Ind_{L\times H\times G}^{L\times H\times L\times G}\Inf_{H\times G}^{L\times H\times G}\big)(\alpha)$.

It is then natural to ask if the assignment $\times:\cp_A\times\cp_A\to \cp_A$ sending $(G,K)$ to $G\times K$ and $(\alpha,\beta)\in A(H\times G)\times A(L\times K)$ to $(\alpha\times L)\circ (G\times \beta)\in A(H\times L\times G\times K)$ is a functor. We will answer this question at the end of Section~\ref{commutant} (Corollary~\ref{PA monoidal}).
\end{rem}

\section{Adjoint functors}
Let $A$ and $C$ be Green $\CD$-biset functors. A morphism $f:A\to C$ of Green $\CD$-biset functors induces an obvious functor $\CP_f:\CP_A\to \CP_C$, which is the identity on objects, and maps $\alpha\in \Hom_{\cp_A}(G,H)=A(H\times G)$ to $f_{H\times G}(\alpha)\in C(H\times G)=\Hom_{\cp_C}(G,H)$. \par
Let $L$ be a fixed group in $\CD$. The {\em inflation} morphism $\Inf_L:A\to A_L$, introduced in~\cite{thbenja}, is the morphism of Green biset functors defined for each $G\in\CD$ and each $\alpha\in A(G)$ by $\Inf_L(\alpha)=A(\Inf_{G}^{G\times L})(\alpha)\in A(G\times L)=A_L(G)$, where $G$ is identified with $(G\times L)/(\{1\}\times L)$. The corresponding functor $\cp_A\to \cp_{A_L}$ will be denoted by $\psi_L$. Explicitely, for each $G\in \CD$, we have $\psi_L(G)=G$, and for a morphism $\alpha\in A(H\times G)$, we have 
$$\psi_L(\alpha)=A(\Inf_{H\times G}^{H\times G\times L})(\alpha)\in A(H\times G\times L)=A_L(H\times G)=\Hom_{\CP_{A_L}}\big(\psi_L(G),\psi_L(H)\big).$$
We introduce another functor $\theta_L:\cp_{A_L}\to \cp_A$, defined as follows:
 for an object $G$ of~$\CP_{A_L}$, wet set $\theta_L(G)=G\times L$, wiewed as an object of~$\CP_A$. For a morphism $\alpha\in \Hom_{\CP_{A_L}}(G,H)=A_L(H\times G)=A(H\times G\times L)$, we define
$$\theta_L(\alpha)=A(\Ind_{H\times G\times L}^{H\times L\times G\times L})(\alpha)\in A(H\times L\times G\times L)=\Hom_{\CP_A}\big(\theta_L(G),\theta_L(H)\big),$$
where $H\times G\times L$ is viewed as a subgroup of $H\times L\times G\times L$ via the injective group homomorphism $(h,g,l)\in H\times G\times L\mapsto (h,l,g,l)\in H\times L\times G\times L$.
\begin{nota} In what follows, we will use a convenient abuse of notation, and generally drop the symbols $\times$ of cartesian products of groups, writing e.g. $HLGL$ instead of $H\times L\times G\times L$.
\end{nota}
\pagebreak[3]
\begin{teo} \label{adjonctions-enonce}
\begin{enumerate}
\item $\psi_L$ is an $R$-linear functor from $\CP_A$ to $\CP_{A_L}$.
\item $\theta_L$ is an $R$-linear functor from $\CP_{A_L}$ to $\CP_A$.
\item The functors $\psi_L$ and $\theta_L$ are left and right adjoint to one another. In other words, for any $G$ and $H$ in $\CD$, there are $R$-module isomorphisms
\begin{align*}
\Hom_{\CP_{A_L}}\big(G,\psi_L(H)\big)&\cong \Hom_{\CP_A}\big(\theta_L(G),H\big)\\
\Hom_{\CP_{A_L}}\big(\psi_L(G),H\big)&\cong \Hom_{\CP_{A}}\big(G,\theta_L(H)\big)\\
\end{align*}
\end{enumerate}
\vspace{-2ex}
which are natural in $G$ and $H$.
\end{teo}
\begin{proof} Point (1) is clear, since the functor $\psi_L$ is built from a morphism of Green biset functors $\Inf_L:A\to A_L$. \mpn
To prove (2), 
let $G, H, K\in\CD$. If $\alpha\in A_L(HG)$ and $\beta\in A_L(KH)$, then
\begin{align*}
\theta_L(\beta)\circ\theta_L(\alpha)&=A(\Def_{KLGL}^{KLHLGL}\Res_{KLHLGL}^{KLHLHLGL})\Big(A\big(\Ind_{KHL}^{KLHL}\big)(\beta)\times A\big(\Ind_{HGL}^{HLGL}\big)(\alpha)\Big)\\
&= A(\Def_{KLGL}^{KLHLGL}\Res_{KLHLGL}^{KLHLHLGL}\Ind_{KHLHGL}^{KLHLHLGL})(\beta\times\alpha).
\end{align*}
In the restriction $\Res_{KLHLGL}^{KLHLHLGL}$, the group $KLHLGL$ maps into $KLHLHLGL$ via 
$$f:(k,l_1,h,l_2,g,l_3)\in KLHLGL\mapsto (k,l_1,h,l_2,h,l_2,g,l_3)\in KLHLHLGL,$$
and in the induction $\Ind_{KHLHGL}^{KLHLHLGL}$, the group $KHLHGL$ maps into $KLHLHLGL$ via 
$$f':(k',h'_1,l'_1,h'_2,g,l'_2)\in KHLHGL\mapsto (k',l'_1,h'_1,l'_1,h'_2,l'_2,g',l'_2)\in KLHLHLGL.$$
Then one checks readily that $\Im(f)\Im(f')=KLHLHLGL$, and that $\Im(f)\cap\Im(f')$ is isomorphic to $KHGL$. Hence by the Mackey formula, there is an isomorphism of bisets
$$\Res_{KLHLGL}^{KLHLHLGL}\Ind_{KHLHGL}^{KLHLHLGL}\cong \Ind_{KHGL}^{KLHLGL}\Res_{KHGL}^{KHLHGL},$$
where in $\Ind_{KHGL}^{KLHLGL}$, the inclusion $KHGL\hookrightarrow KLHLGL$ is $(k,h,g,l)\mapsto (k,l,h,l,g,l)$, and in $\Res_{KHGL}^{KHLHGL}$, the inclusion $KHGL\hookrightarrow KHLHGL$ is $(k,h,g,l)\mapsto (k,h,l,h,g,l)$.\par
Now in the deflation $\Def_{KLGL}^{KLHLGL}$, the group $KLHLGL$ maps onto $KLGL$ via $(k,l_1,h,l_2,g,l_3)\mapsto (k,l_1,g,l_3)$. It follows that there is an isomorphism of bisets
$$\Def_{KLGL}^{KLHLGL}\Ind_{KHGL}^{KLHLGL}\cong \Ind_{KGL}^{KLGL}\Def_{KGL}^{KHGL},$$
which gives
\begin{align*}
\theta_L(\beta)\circ\theta_L(\alpha)&=A(\Ind_{KGL}^{KLGL})A(\Def_{KGL}^{KHGL}\Res_{KHGL}^{KHLHGL})(\beta\times\alpha)\\
&=A(\Ind_{KGL}^{KLGL})A_L(\Def_{KG}^{KHG}\Res_{KHG}^{KHHG})A(\Res_{KHHGL}^{KHLHGL})(\beta\times\alpha)\\
&=A(\Ind_{KGL}^{KLGL})A_L(\Def_{KG}^{KHG}\Res_{KHG}^{KHHG})(\beta\times^d\alpha)\\
&=A(\Ind_{KGL}^{KLGL})(\beta\circ^d\alpha)=\theta_L(\beta\circ^d\alpha),
\end{align*}
where $\circ^d$ denotes the composition in the category $\CP_{A_L}$.
This shows that $\theta_L$ is compatible with composition of morphisms. A straightforward computation shows that it maps identity morphisms to identity morphisms. This completes the proof of Assertion 2, since $\theta_L$ is obviously $R$-linear.\mpn
(3) Since the complete proof of Assertion~3 demands the verification of many technical details, we only include the full proof that $\theta_L$ is left adjoint to $\psi_L$. We next simply give the description of the bijection involved in the other direction, and leave the corresponding verifications to the reader.\par
For $G$ and $H$ in $\CD$, we have
$$\Hom_{\CP_{A_L}}\big(G,\psi_L(H)\big)=A_L\big(\psi_L(H)G\big)=A(HGL)\;\;\hbox{and}\;\;\Hom_{\CP_A}\big(\theta_L(G),H\big)=A(HGL),$$
so an obvious candidate for an isomorphism $\Hom_{\CP_{A_L}}\big(G,\psi_L(H)\big)\to \Hom_{\CP_A}\big(\theta_L(G),H\big)$ is the identity map of $A(HGL)$. To avoid confusion, for $\alpha\in \Hom_{\CP_{A_L}}\big(G,\psi_L(H)\big)$, we denote by $\widetilde{\alpha}$ the element $\alpha$ viewed as an element of $\Hom_{\CP_A}\big(\theta_L(G),H\big)$.\par
We now check that the map $\alpha\mapsto \widetilde{\alpha}$ is natural in $G$ and $H$. For naturality in $G$, if $G'\in\CD$ and $u\in \Hom_{A_L}(G',G)$, we have the diagrams
$$\xymatrix{
G\ar[r]^-\alpha&\psi_L(H)\\
G'\ar[u]^-u\ar[ur]_-{\alpha u}
}\hspace{4ex}
\xymatrix{
\theta_L(G)\ar[r]^-{\widetilde{\alpha}}&H\\
\theta_L(G')\ar[u]^-{\theta_L(u)}\ar[ur]_-{\widetilde{\alpha \circ^du}}
}
$$
and we have to show that the right-hand side diagram is commutative, i.e. that $\widetilde{\alpha\circ^d u}=\widetilde{\alpha}\circ\theta_L(u)$. But
\begin{align*}
\widetilde{\alpha}\circ\theta_L(u)&=\alpha\circ A(\Ind_{GG'L}^{GLG'L})(u)\\
&=A(\Def_{HG'L}^{HGLG'L}\Res_{HGLG'L}^{HGLGLG'L})\Big(\alpha\times A\big(\Ind_{GG'L}^{GLG'L}\big)(u)\Big)\\
&=A(\Def_{HG'L}^{HGLG'L}\Res_{HGLG'L}^{HGLGLG'L}\Ind_{HGLGG'L}^{HGLGLG'L})(\alpha\times u).
\end{align*}
In the restriction $\Res_{HGLG'L}^{HGLGLG'L}$, the inclusion $HGLG'L\hookrightarrow HGLGLG'L$ is the map
$$f:(h,g,l_1,g',l_2)\in HGLG'L\mapsto (h,g,l_1,g,l_1,g',l_2)\in HGLGLG'L,$$
and in the induction $\Ind_{HGLGG'L}^{HGLGLG'L}$, the inclusion $HGLGG'L\hookrightarrow HGLGLG'L$ is the map
$$f':(\eta,\gamma_1,\lambda_1,\gamma_2,\gamma',\lambda_2)\in HGLGG'L\mapsto (\eta,\gamma_1,\lambda_1,\gamma_2,\lambda_2,\gamma',\lambda_2)\in HGLGLG'L.$$
Then clearly $\Im(f)\Im(f')=HGLGLG'L$, and $\Im(f)\cap\Im(f')\cong HGG'L$. By the Mackey formula, this gives an isomorphism of bisets
$$\Res_{HGLG'L}^{HGLGLG'L}\Ind_{HGLGG'L}^{HGLGLG'L}\cong \Ind_{HGG'L}^{HGLG'L}\Res_{HGG'L}^{HGLGG'L},$$
where, in $\Ind_{HGG'L}^{HGLG'L}$, the inclusion $HGG'L\hookrightarrow HGLG'L$ is $(h,g,g',l)\mapsto (h,g,l,g',l)$, and in $\Res_{HGG'L}^{HGLGG'L}$, the inclusion $HGG'L\hookrightarrow HGLGG'L$ is $(h,g,g',l)\mapsto (h,g,l,g,g',l)$.\par
Now in $\Def_{HG'L}^{HGLG'L}$, the quotient map $HGLG'L\to HG'L$ sends $(h,g,l_1,g',l_2)$ to $(h,g',l_2)$, so the image of the subgroup $HGG'L$ is the whole of $HG'L$. It follows that there is an isomorphism of bisets
$$\Def_{HG'L}^{HGLG'L}\Ind_{HGG'L}^{HGLG'L}\cong\Def_{HG'L}^{HGG'L},$$
which gives finally
\begin{align*}
\widetilde{\alpha}\circ\theta_L(u)&=A(\Def_{HG'L}^{HGG'L}\Res_{HGG'L}^{HGLGG'L})(\alpha\times u)\\
&=A_L(\Def_{HG'}^{HGG'})A_L(\Res_{HGG'}^{HGGG'})A(\Res_{HGGG'L}^{HGLGG'L})(\alpha\times u)\\
&=A_L(\Def_{HG'}^{HGG'})A_L(\Res_{HGG'}^{HGGG'})(\alpha\times^du)\\
&=\alpha\circ^d u,
\end{align*}
as was to be shown.\par
We now check that the map $\alpha\mapsto\widetilde{\alpha}$ is natural in $H$. If $H'\in\CD$ and $v\in\Hom_{\CP_A}(H,H')=A(H'H)$, we have the diagrams
$$\xymatrix{
A\ar[r]^-\alpha\ar[dr]_-{\psi_L(v)\circ^d\alpha}&\psi_L(H)\ar[d]^-{\psi_L(v)}\\
&\psi_L(H')
}\hspace{4ex}
\xymatrix{
\theta_L(G)\ar[r]^-{\widetilde{\alpha}}\ar[dr]_-{\widetilde{\psi_L(v)\circ^d\alpha}}&H\ar[d]^-v\\
&H'
}
$$
and we have to show that the right-hand side diagram is commutative, i.e. that $\widetilde{\psi_L(v)\circ^d\alpha}=v\circ\widetilde{\alpha}$. But
\begin{align*}
\psi_L(v)\circ^d\alpha&=A(\Inf_{H'H}^{H'HL})(v)\circ^d\alpha=A_L(\Def_{H'G}^{H'HG}\Res_{H'HG}^{H'HHG})\Big(A(\Inf_{H'H}^{H'HL})(v)\times^d\alpha\Big)\\
&=A(\Def_{H'GL}^{H'HGL}\Res_{H'HGL}^{H'HHGL})A(\Res_{H'HHGL}^{H'HLHGL}\Inf_{H'HHGL}^{H'HLHGL})(v\times\alpha)\\
&=A(\Def_{H'GL}^{H'HGL}\Res_{H'HGL}^{H'HLHGL}\Inf_{H'HHGL}^{H'HLHGL})(v\times\alpha).
\end{align*}
In $\Res_{H'HGL}^{H'HLHGL}$, the inclusion $H'HGL\hookrightarrow H'HLHGL$ is $(h',h,g,l)\mapsto (h',h,l,h,g,l)$, and in $\Inf_{H'HHGL}^{H'HLHGL}$, the quotient map $H'HLHGL\to H'HHGL$ is $(h',h_1,l_1,h_2,g,l_2)\mapsto (h',h_1,h_2,g,l_2)$. The composition of these two maps sends $(h',h,g,l)$ to $(h',h,h,g,l)$, hence it is injective. This gives an isomorphim of bisets
$$\Res_{H'HGL}^{H'HLHGL}\Inf_{H'HHGL}^{H'HLHGL}\cong \Res_{H'HGL}^{H'HHGL},$$
from which we get
$$ \psi_L(v)\circ^d\alpha=A(\Def_{H'GL}^{H'HGL}\Res_{H'HGL}^{H'HHGL})(v\times\alpha)=v\circ\widetilde{\alpha},$$
as was to be shown. \par
Hence the isomorphism $\alpha\in \Hom_{\CP_{A_L}}\big(G,\psi_L(H)\big)\mapsto \widetilde{\alpha}\in\Hom_{\CP_A}\big(\theta_L(G),H\big)$ is natural in $G$ and $H$, so $\theta_L$ is left adjoint to $\psi_L$.\par
We now describe the bijection implying that $\theta_L$ is also right adjoint to $\psi_L$. So, for $G,H\in\CD$, we have to build an isomorphism
$$\alpha\in\Hom_{\CP_{A_L}}\big(\psi_L(G),H\big)\mapsto \widehat{\alpha}\in \Hom_{\CP_A}\big(G,\theta_L(H)\big)$$
of $R$-modules, natural in $G$ and $H$. But
$$\Hom_{\CP_{A_L}}\big(\psi_L(G),H\big)=A_L(HG)=A(HGL)\;\hbox{and}\;\Hom_{\CP_A}\big(G,\theta_L(H)\big)=A(HLG),$$
so an obvious candidate for the above isomorphism is to set $\widehat{\alpha}=A(\Iso_{HGL}^{HLG})(\alpha)$. The verification that this isomorphism is functorial in $G$ and $H$ is similar to the proof of the first adjunction, and we omit it. 
\end{proof}

\begin{defi} Let $A$ be a Green $\CD$-biset functor and $L\in\CD$. We denote by $$\Psi_L:{\rm Fun}_R(\CP_{A_L},\AMod{R})\to {\rm Fun}_R(\CP_{A},\AMod{R})$$
the functor induced by precomposition with $\psi_L$, and by 
$$\Theta_L:{\rm Fun}_R(\CP_{A},\AMod{R})\to {\rm Fun}_R(\CP_{A_L},\AMod{R})$$
the functor induced by precomposition with $\theta_L$.
\end{defi}
\begin{prop} The functors $\Psi_L$ and $\Theta_L$ are mutual left and right adjoint functors between ${\rm Fun}_R(\CP_{A_L},\AMod{R})$ and ${\rm Fun}_R(\CP_{A},\AMod{R})$.
\end{prop}
\begin{proof} This follows from Theorem~\ref{adjonctions-enonce}, by standard category theory.
\end{proof}

\begin{rem} Using the above equivalences of categories between ${\rm Fun}_R(\CP_{A},\AMod{R})$ and $\AMod{A}$, and ${\rm Fun}_R(\CP_{A_L},\AMod{R})$ and $\AMod{A_L}$, we will consider $\Psi_L$ as a functor from $\AMod{A_L}$ to $\AMod{A}$ and $\Theta_L$ as a functor from $\AMod{A}$ to $\AMod{A_L}$. One can check that, from this point of view, if $N$ is an $A_L$-module, then $\Psi_L(N)$ is the $A$-module defined as follows:
\begin{itemize}
\item If $G\in\CD$, then $\Psi_L(N)(G)=N(G)$.
\item If $G,H\in\CD$, $a\in A(G)$ and $v\in N(H)$, then
$$a\times v=A(\Inf_G^{G\times L})(a)\times^dv\,$$
where $\times^d$ denotes the action of $A_L$ on $N$, and $A(\Inf_G^{G\times L})(a)\in A(G\times L)$ is viewed as an element of $A_L(G)$.
\end{itemize}
Conversely, if $M$ is an $A$-module, then $\Theta_L(M)$ is the $A_L$-module defined as follows:
\begin{itemize}
\item If $G\in\CD$, then $\Theta_L(M)(G)=M(G\times L)$.
\item If $G,H\in\CD$, $a\in A_L(G)$ and $m\in M(H\times L)$, then
$$a\times^dm=M(\Res_{G\times H\times L}^{G\times L\times H\times L})(a\times m),$$
where $a\times m$ is the product of $a\in A(G\times L)$ and $m\in M(H\times L)$, and $H\times G\times L$ is viewed as a subgroup of $G\times L\times H\times L$ via the map $(g,h,l)\mapsto (g,l,h,l)$.
\end{itemize}
\end{rem}

\begin{teo} Let $A$ be a Green $\CD$-biset functor, and $L\in \CD$. The endofunctor $\rho_L$ of~$\cp_A$ is isomorphic to $\theta_L\circ\psi_L$ and so the endofunctor $\Psi_L\circ\Theta_L$ of $\AMod{A}$ is isomorphic to the Yoneda-Dress functor $\Id_L$. In particular $\Id_L$ is self adjoint.
\end{teo}
\begin{proof} One checks readily that $\rho_L$ is isomorphic to the composition $\theta_L\circ\psi_L$. The other assertions follow by Theorem~\ref{adjonctions-enonce}, as the Yoneda-Dress functor $\Id_L$ is obtained by precomposition with $\rho_L={-}\times L$.
\end{proof}

We observe that the $L$-shift of the $A$-module $A$ is the representable functor $A(-,L)$ of the category $\cp_A$, so it is projective. More generally, the $L$-shift of the representable functor $A(-,X)$ is the representable functor $A(-,L\times X)$. Hence the Yoneda-Dress construction maps a representable functor to a representable functor. 

\section{The commutant}\label{commutant}

\begin{defi}
Let $A$ be a Green {$\CD$-}biset functor. 
\begin{enumerate}
\item For $G,H\in\CD$, we say that an element $a\in A(G)$ and an element $b\in A(H)$ {\em commute} if 
$$a\times b=A(\Iso_{H\times G}^{G\times H})(b\times a).$$
\item For a group $G$ {in $\CD$}, we denote by $CA(G)$ the set of elements of $A(G)$ which commute with any element of $A(H)$, for any $H\in\CD$, i.e.
\begin{displaymath}
\{a\in A(G)\mid {\forall H\in\CD, \,\forall b\in A(H), \;a\times b=A(\Iso_{H\times G}^{G\times H})(b\times a)}\},
\end{displaymath}
and call it {\em the commutant} of $A$ {at} $G$.
\end{enumerate}
\end{defi}

Observe that $CA(G)$ is an $R$-submodule of $A(G)$, since the product $\times$ is bilinear.

\begin{lema} Let $A$ be a Green {$\CD$-}biset functor. Then the commutant of $A$ is a Green {$\CD$-}biset subfunctor of $A$.
\end{lema}
\begin{proof}
To see it is a biset functor, let $Y$ be a $(K,\, G)$-biset {for groups $K$ and $G$ in $\CD$}, and $a$ be in $CA(G)$. If $b$ is in $A(H)$ for a given group $H$ {in $\CD$}, we have
\begin{displaymath}
A(Y)(a)\times b=A\left((Y\times H)\circ\Iso_{H\times G}^{G\times H}\right)(b\times a)
\end{displaymath}
where $Y\times H$ is seen as a $(K\times H,\, G\times H)$-biset. If we show that $(Y\times H)\circ\Iso_{H\times G}^{G\times H}$ is isomorphic to $\Iso_{H\times K}^{K\times H}\circ (H\times Y)$, where $H\times Y$ is seen as a $(H\times K,\, H\times G)$-biset, the right-hand side of the equality above will be equal to 
\begin{displaymath}
A(\Iso_{H\times K}^{K\times H})\big(b\times A(Y)(a)\big),
\end{displaymath}
which is what we want. Now, $\Iso_{H\times G}^{G\times H}$ is the group $H\times G$, seen as a $(G\times H,\, H\times G)$-biset, and $\Iso_{H\times K}^{K\times H}$ is the group $H\times K$, seen as a $(K\times H,\, H\times K)$-biset. So, it is not hard to see that $(Y\times H)\circ\Iso_{H\times G}^{G\times H}$ is isomorphic to $Y\times H$ as $(K\times H,\, H\times G)$-biset, where the right action of $H\times G$ is given by $(y,\, h)(h_1,\, g_1)=(yg_1,\, hh_1)$. Similarly, $\Iso_{H\times K}^{K\times H}\circ (H\times Y)$ is isomorphic to $H\times Y$ as $(K\times H,\, H\times G)$-set, where the left action of $K\times H$ is given by $(k_1,\, h_1)(h,\, y)=(h_1h,\, k_1y)$. Hence, it is easy to verify that the map $Y\times H\rightarrow H\times Y$ sending $(y,\, h)$ to $(h,\, y)$ defines an isomorphism between these two bisets.

To see that $CA$ is closed under the product $\times$, let $a$ be in $CA(G)$, $b$ be in $CA(H)$ and $c$ be in $A(K)$. We have
\begin{displaymath}
a\times (b\times c)=a\times A(\Iso_{K\times H}^{H\times K})(c\times b), 
\end{displaymath}
which is clearly equal to $A(\Iso_{G\times K\times H}^{G\times H\times K})(a\times c\times b)$. Similarly
\begin{displaymath}
(a\times c)\times b=A(\Iso_{K\times G\times H}^{G\times K\times H})(c\times a\times b).
\end{displaymath}
 Finally, clearly we have
\begin{displaymath}
\Iso_{G\times K\times H}^{G\times H\times K}\circ\Iso_{K\times G\times H}^{G\times K\times H}= 
\Iso_{K\times G\times H}^{G\times H\times K},
\end{displaymath} 
 which {yields} the first equality
\begin{displaymath}
(a\times b)\times c=A(\Iso_{K\times G\times H}^{G\times H\times K})\left(c\times (a\times b)\right).
\end{displaymath} 
To finish the proof, it is clear that the identity element $\varepsilon\in A(1)$ belongs to $CA(1)$.
\end{proof}

\begin{coro} \label{bisets commute}Let $A$ be a Green $\CD$-biset functor. Then the image of the unique Green biset functor morphism $\upsilon_A:RB\to A$ is contained in $CA$.
\end{coro}
\begin{proof} Indeed, by uniqueness of $\upsilon_A$ and $\upsilon_{CA}$, the diagram
$$\xymatrix@C=3ex@R=3ex{
&\makebox[0pt]{\raisebox{1ex}{$CA$}}\ar@{^{(}->}[dr]\;&\\
RB\ar[ur]^-{\upsilon_{CA}}\ar[rr]_-{\upsilon_A}&&A
}
$$ is commutative.
\end{proof}

\begin{defi}
We will say that a Green $\CD$-biset functor $A$ is {\em commutative} if $A=CA$.
\end{defi}

 It is easy to see that $CA$ is commutative. All the examples considered in Example~\ref{ejemplos} are commutative Green biset functors.

If $A$ is commutative, then clearly $A_G$ is commutative for any $G${. More} generally we have the following result.

\begin{prop}
\label{commsft}
Let $A$ be a Green {$\CD$-}biset functor and ${G\in\CD}$. {Then} $CA_G=(CA)_G$.
\end{prop}
\begin{proof}
Observe that $CA_G$ and $(CA)_G$ are both Green {$\CD$-}biset subfunctors of $A_G$, so to prove they are equal as Green {$\CD$-}biset functors, it suffices to prove that for every group ${H\in\CD}$, we have $(CA)_G(H){=} CA_G(H)$. \par
{To prove that $(CA)_G(H)\subseteq CA_G(H)$, we choose a group $K$ in $\CD$, and elements $a\in(CA)_G(H)$ and $b\in A_G(K)$. We} must prove that
\begin{displaymath}
a\times^d b=A_G(\Iso_{K\times H}^{H\times K})(b\times^da). 
\end{displaymath}
{We have}
\begin{displaymath}
a\times^db=A\left(\Res^{H\times G\times K\times G}_{H\times K\times \Delta(G)}\right)(a\times b)\quad\textrm{and}\quad b\times^da=A\left(\Res^{K\times G\times H\times G}_{K\times H\times \Delta(G)}\right)(b\times a).
\end{displaymath}
Now, by definition $(CA)_G(H)=CA(H\times G)$,  so the element $a$ satisfies
\begin{displaymath}
a\times b=A\left(\Iso_{K\times G\times H\times G}^{H\times G\times K\times G}\right)(b\times a).
\end{displaymath}
Substituting this in the above equation on the left we easily obtain what we wanted.

{To prove the reverse inclusion $CA_G(H)\subseteq (CA)_G(H)$, we now let $a\in CA_G(H)$ and $b\in A(K)$,} and consider $c=A(\Inf_K^{K\times G})(b)$. Then we have 
\begin{displaymath}
a\times^d c=A_G(\Iso_{K\times H}^{H\times K})(c\times^da),
\end{displaymath}
and clearly
\begin{displaymath}
a\times^dc=A\left(\Res^{H\times G\times K\times G}_{H\times K\times \Delta(G)}\circ\Inf_{H\times G\times K}^{H\times G\times K\times G}\right)(a\times {b}).
\end{displaymath}
But it is easy to see (for example from Section 1.1.3 of \cite{biset}) that
\begin{displaymath}
\Res^{H\times G\times K\times G}_{H\times K\times \Delta(G)}\circ\Inf_{H\times G\times K}^{H\times G\times K\times G}\cong \Iso_{H\times G\times K}^{H\times K\times \Delta(G)}.
\end{displaymath}
By doing a similar transformation with $c\times^da$, and applying the corresponding isomorphisms, we easily obtain what we wanted.
\end{proof}

\begin{lema}
Let $A$ be a Green $\CD$-biset functor. Then for any group $G$ {in $\CD$}, the commutant $CA(G)$ is a subring of $Z{\big(}A(G){\big)}$.
\end{lema}
\begin{proof}
Take $a\in CA(G)$ and $b\in A(G)$, then
\begin{eqnarray*}
a \cdot b &=& A\left(\Iso_{\Delta(G)}^G\circ\Res^{G\times G}_{\Delta(G)}\right)(a\times b)\\
 & = & A\left(\Iso_{\Delta(G)}^G\circ\Res^{G\times G}_{\Delta(G)}\circ{\Iso(\sigma_G)}\right)(b\times a)\\
 & = & A\left(\Iso_{\Delta(G)}^G\circ\Res^{G\times G}_{\Delta(G)}\right)(b\times a)=b \cdot a,
\end{eqnarray*}
{where $\sigma_G$ is the automorphism of $G\times G$ switching the components.} Since $CA(G)$ and $Z(A(G))$ have the same ring structure, inherited from the Green {$\CD$-}biset functor structure of $A$, this shows that $CA(G)$ is a subring of $Z{\big(}A(G){\big)}$.
\end{proof}

\begin{rem} It is not hard to see then that $A$ is a commutative Green biset functor if and only if for every group $G$, the ring $A(G)$ is a commutative ring.
\end{rem}
We now answer the question raised in Remark~\ref{bifunctor}.
\begin{prop} \label{commute} Let $A$ be a Green $\CD$-biset functor, and $G,H,K,L\in \CD$. Let $\alpha\in A(HG)$ and $\beta\in A(LK)$. Then the square
$$\xymatrix{
G\times K\ar[r]^-{G\times\beta}\ar[d]_-{\alpha\times K}&G\times L\ar[d]^{\alpha\times L}\\
H\times K\ar[r]_-{H\times\beta}&H\times L
}
$$
commutes in $\cp_A$ if and only if $\alpha$ and $\beta$ commute.
\end{prop}
\begin{proof} Let $u=(\alpha\times L)\circ(G\times\beta)$. By definition
\begin{align*}
u&= A(\Ind_{HGL}^{HLGL}\Inf_{HG}^{HGL})(\alpha)\circ A(\Ind_{GLK}^{GLGK}\Inf_{LK}^{GLK})(\beta)\\
&= A(\Def_{HLGK}^{HL\sou{GL}GK}\Res_{HL\sou{GL}GK}^{HL\sou{GLGL}GK})\big(A(\Ind_{HG\sou{L}}^{H\sou{L}G\sou{L}}\Inf_{HG}^{HG\sou{L}})(\alpha)\times A(\Ind_{\sou{G}LK}^{\sou{G}L\sou{G}K}\Inf_{LK}^{\sou{G}LK})(\beta)\big),
\end{align*}
where the notation $\Def_{HLGK}^{HL\sou{GL}GK}$ means the deflation with respect to the underlined normal subgroup, and  $\Res_{HL\sou{GL}GK}^{HL\sou{GLGL}GK}$ means that the underlined $GL$ in subscript embeds diagonally in the underlined $GLGL$ in superscript. Similarly in $\Ind_{HG\sou{L}}^{H\sou{L}G\sou{L}}$, the group $L$ in subscript embed diagonally in the two underlines copies of $L$ in superscript, and in $\Inf_{HG}^{HG\sou{L}}$, inflation is relative to the underlined $L$ in superscript. Thus
\begin{align*}
u&=A(\Def_{HLGK}^{HL\sou{GL}GK}\Res_{HL\sou{GL}GK}^{HL\sou{GLGL}GK}\Ind_{HG\sou{L}\sou{G}LK}^{H\sou{L}G\sou{L}\sou{G}L\sou{G}K}\Inf_{HGLK}^{HG\sou{LG}LK})(\alpha\times\beta)
\end{align*}
Standard relations in the composition of bisets (see Section~1.1.3 and Lemma~2.3.26  of~\cite{biset})
and some tedious but straightforward calculations finally give
$$u=(\alpha\times L)\circ(G\times\beta)=A(\Iso^{HLGK}_{HGLK})(\alpha\times\beta).$$
Similar calculations show that
$$(H\times\beta)\circ(\alpha\times K)=A(\Iso^{HLGK}_{LKHG})(\beta\times\alpha).$$
So $(H\times\beta)\circ(\alpha\times K)=(\alpha\times L)\circ(G\times\beta)$ if and only if 
\begin{align*}
\beta\times\alpha&=A(\Iso_{HLGK}^{LKHG}\Iso^{HLGK}_{HGLK})(\alpha\times\beta)\\
&=A(\Iso_{HGLK}^{LKHG})(\alpha\times\beta),
\end{align*}
that is, if $\alpha$ and $\beta$ commute.  
\end{proof}

\begin{coro} \label{PA monoidal}The assignment $\times:\cp_A\times\cp_A\to \cp_A$ sending $(G,K)$ to $G\times K$ and $(\alpha,\beta)\in A(H\times G)\times A(L\times K)$ to $(\alpha\times L)\circ (G\times \beta)\in A(H\times L\times G\times K)$ is a functor if and only if $A$ is commutative. In particular, when $A$ is commutative, this functor $\times$ endows $\cp_A$ with a structure of a {\em symmetric monoidal category}.
\end{coro}

\section{The center}
\begin{defi}
{Let $A$ be a Green $\CD$-biset functor.} For a group $L$ in $\CD$, we denote by $ZA(L)$ the family of all natural transformations {$\Id\to \Id_L$} from the identity functor $Id:\AMod{A}\rightarrow \AMod{A}$ to the functor $\Id_L$. We call it {{\em the center}} of $A$ at $L$.
\end{defi}

When $L$ is trivial, the functor $\Id_L$ is isomorphic to the identity functor, hence $ZA(1)$ is the family of natural endotransformations of the identity functor. So our definition is analogous to that of the center of a category (see for example Hoffmann~\cite{hoff}  for arbitrary categories, or Section 19 of Butler-Horrocks~\cite{butler-horrocks} for abelian categories). Nonetheless, we want to regard this center as a Green {$\CD$-}biset functor, and see its relation with the commutant $CA$. Our construction is inspired by an analogous construction for Green functors over a fixed finite group in~\cite{bGfun} Section 12.2.

\subsection{The center as a Green biset functor}
Our goal is to show that for each Green $\CD$-biset functor $A$, the assignment $L\mapsto ZA(L)$ is itself a Green $\CD$-biset functor. For this, we will first give an equivalent description of $ZA(L)$, and then build a Green functor structure on $ZA$.
\begin{prop} \label{ZAZA'}Let $A$ be a Green $\CD$-biset functor, and $L\in\CD$. Then $ZA(L)$ is isomorphic to the family $ZA'(L)$ of natural transformations from the identity functor of $\cp_A$ to~$\rho_L$.
\end{prop}
\begin{proof}
Consider the Yoneda embedding $\CY_A:\cp_A\to \AMod{A}$ sending $L\in\CD$ to the functor $A(-,L)$. Since $\Id_L$ preserves the image of $\CY_A$, which is a fully faithful functor, we have $\Id_L\circ \CY_A=\CY_A\circ \rho_L$, and it follows that each element of $ZA(L)$ induces a natural transformation from the identity functor of $\cp_A$, denoted by $\rho_1$, to $\rho_L$. In this way, we get a linear map $f_L:ZA(L)\to ZA'(L)$. Conversely, each natural transformation $\rho_1\to \rho_L$  induces a natural transformation $\CY_A\to \Id_L\circ \CY_A$. Since the image of $\CY_A$ generates $\AMod{A}$, such a natural transformation extends to a natural transformation from the identity functor of $\AMod{A}$ to $\Id_L$. This gives a linear map $g_L:ZA'(L)\to ZA(L)$. Clearly $f_L$ and $g_L$ are inverse to one another. 
\end{proof}

We will now use the previous identification to get a better understanding of $ZA(L)$. Indeed, a natural transformation~$t$ from the identity functor of $\cp_A$ to the functor $\rho_L={-}\times L=\theta_L\psi_L$ consists, for each $G\in \CD$, of a morphism $t_G:G\to G\times L$ in $\cp_A$, i.e. $t_G\in A(G\times L\times G)$, such that for any $H\in \CD$ and any $\alpha\in A(H\times G)$, the diagram
\begin{equation}\label{diagram Z}
\vcenter{\xymatrix{
G\ar[r]^-{t_G}\ar[d]_-\alpha&G\times L\ar[d]^-{\alpha\times L=\theta_L\psi_L(\alpha)}\\
H\ar[r]^-{t_H}&H\times L
}}
\end{equation}
is commutative in $\cp_A$.

\begin{lema}
Let $G,H\in \CD$, and $\alpha\in A(H\times G)=\Hom_{\cp_A}(G,H)$. For an element $u$ of $A(H\times L\times G)=\Hom_{\cp_A}(G,H\times L)$, let $u^\natural$ denote the element $u$, viewed as a morphism from $L\times G$ to $H$ in $\cp_A$. Then for any $t\in A(G\times L\times G)$
$$\big(\theta_L\psi_L(\alpha)\circ t\big)^\natural=\alpha\circ t^\natural \;\;\hbox{in}\;\; A(H\times L\times G).$$
\end{lema}
\begin{proof}
The functor $\rho_L$ is a self-adjoint $R$-linear endofunctor of $\cp_A$. It follows from the proof of Theorem~\ref{adjonctions-enonce} that for any $G,H\in\cp_A$, the natural bijection given by this adjunction
$$v\in\Hom_{\cp_A}\big(G,\rho_L(H)\big)=A(HLG)\to v^\sharp\in\Hom_{\cp_A}\big(\rho_L(G),H\big)=A(HGL)$$
is induced by the isomorphism $HLG\to HGL$ switching the components $L$ and $G$. By adjunction we have commutative diagrams
$$\xymatrix{
G\ar[r]^-t\ar[rd]_-{\rho_L(\alpha)\circ t}&\rho_L(G)\ar[d]^-{\rho_L(\alpha)}\\
&\rho_L(H)
}
\hspace{4ex}
\xymatrix{
\rho_L(G)\ar[r]^-{t^\sharp}\ar[dr]_{\big(\rho_L(\alpha)\circ t\big)^\sharp}&G\ar[d]^-\alpha\\
&H
}
$$
so $\big(\rho_L(\alpha)\circ t\big)^\sharp=\alpha\circ t^\sharp$. Since $t^\natural=t^\sharp\circ\tau_{L,G}$, where $\tau_{G,L}:LG\to GL$ is the isomorphism switching $G$ and $L$, the lemma follows by right composition of the previous equality with $\tau_{G,L}$.
\end{proof}
Since $v$ and $v^\natural$ are actually the same element of $A(HLG)$, for any $v\in A(HLG)$, the commutativity in Diagram~(\ref{diagram Z}) can be simply written as
\begin{equation}\label{def Cr}\alpha\circ_G t_G=t_H\circ_H \alpha,\end{equation}
where $\circ_G$ is the composition $A(HG)\times A(GLG)\to A(HLG)$, and $\circ_H$ is the composition $A(HLH)\times A(HG)\to A(HLG)$. Thus:
\begin{prop} \label{ZA Cr}Let $A$ be a Green $\CD$-biset functor, and $L\in \CD$. Then an element $t$ of $ZA(L)$ consists of a family of elements $t_G\in A(GLG)$, for every $G\in\CD$, such that $\alpha\circ_G t_G=t_H\circ_H \alpha$, for any $G, H$ in $\CD$ and $\alpha\in A(HG)$. In particular $ZA(L)$ is a set.
\end{prop}

\begin{proof} It remains to see that $ZA(L)$ is a set. This is clear, since an element $t$ of~$ZA(L)$ is determined by its components $t_G$, where $G$ runs trough our chosen set {\bf D} of representatives of isomorphism classes of groups in $\CD$. More precisely $ZA(L)$ is in one to one correspondence with the set $Cr_A(L)$ of sequences of elements $(t_G)_{G\in\mathbf{D}}\in\prod_{G\in \mathbf{D}}\limits A(GLG)$ such that the above condition~(\ref{def Cr}) holds for any $G,H\in\mathbf{D}$ and any $\alpha\in A(HG)$.
\end{proof}

\begin{prop} \label{natural xi} \begin{enumerate}
\item Let $K,L\in\CD$, and $\beta\in CA(LK)$. Then the family of morphisms $\lambda_G(\beta)=G\times\beta:G\times K\to G\times L$, for $G\in\CD$, define a natural transformation of functors $\rho_\beta$ from $\rho_K$ to $\rho_L$.
\item Let $\End_R(\cp_A)$ denote the category of $R$-linear endofunctors of $\cp_A$, where morphisms are natural transformations of functors. Then the assignment
$$\left\{\begin{array}{rcl}K\in\CD&\mapsto& \rho_K\in\End_R(\cp_A)\\
\beta\in CA(LK)&\mapsto&(\rho_\beta:\rho_K\to\rho_L)\end{array}\right.$$
is a faithful $R$-linear functor $\rho_{CA}$ from $\cp_{CA}$ to $\End_R(\cp_A)$.
\end{enumerate}
\end{prop}
\begin{proof} (1) This follows from Proposition~\ref{commute}.\medskip\par\noindent
(2) We have to check that if $G,J,K,L\in\CD$, if $\alpha\in A(KJ)$ and $\beta\in A(LK)$, then $(G\times\beta)\circ(G\times\alpha)=G\times(\beta\circ\alpha)$ in $A(GLGJ)$, and that if $\beta$ is the identity element of $CA(KK)$, then $G\times\beta$ is the identity morphism of $G\times K$ in $\cp_A$. This follows from the fact that $\lambda_G$ is a functor.\par
So we get a functor $\rho_{CA}:\cp_{CA}\to \End_R(\cp_A)$. Seing that this functor is faithful amounts to seing that if $\beta\in CA(LK)$, then $\rho_\beta=0$ if and only if $\beta=0$. But the component $1\times \beta$ of $\rho_\beta$ is clearly equal to $\beta$, after identification of $1\times K$ with $K$ and $1\times L$ with $L$.
\end{proof}
\begin{rem} In particular, it follows from Assertion 2 that an isomorphism of groups $K\to K'$ induces an isomorphism of functors $\rho_K\to\rho_K'$: indeed a group isomorphism $\varphi: K\to K'$ is represented by a $(K',K)$-biset $U_\varphi\in RB(K'K)$, hence by an element $\beta_\varphi=\upsilon_{K'K}(U_\varphi)\in CA(K'K)$, by Corollary~\ref{bisets commute}. The corresponding natural transformation $\rho_{\beta_\varphi}$ is an isomorphism $\rho_K\to\rho_K'$, with inverse $\rho_{\beta_{\varphi^{-1}}}$.
\end{rem}

\begin{lema} Let $A$ be a Green $\CD$-biset functor and $K,L\in\CD$.
\begin{enumerate} \item The endofunctors $\rho_L\circ\rho_K$  and $\rho_{KL}$ of $\cp_A$ are naturally isomorphic.
\item Let $s\in ZA(LK)$, given by the family of elements $s_G\in A(GLKG)$, for $G\in\CD$. Then the natural transformation $s^o:\rho_K\to\rho_L$ deduced from $s:\Id\to \rho_K\rho_L$ by adjunction, is defined by the family of morphisms
$$s^o_G=\Iso^{GLGK}_{GLKG}(s_G)\in A(GLGK)=\Hom_{\cp_A}(GK,GL).$$
\item The map $s\mapsto s^o$ is an isomorphism of $R$-modules 
$$ZA(LK)\to \Hom_{\End_R(\cp_A)}(\rho_K,\rho_L).$$
\end{enumerate}
\end{lema}
\begin{proof}
(1) This follows from a straightforward verification. \medskip\par\noindent
(2) Indeed, by the proof of Theorem~\ref{adjonctions-enonce}, for each $G\in\CD$, the morphism $s_G\in A(GLKG)$
$$s_G:G\to GLK=\rho_K\rho_L(G)=\theta_K\psi_K\rho_L(G)$$
in $\cp_A$ gives by adjunction the morphism
$$u:\psi_K(G)\to \psi_K\rho_L(G),$$
in $\cp_{A_K}$, defined as the element $u=A(\Iso_{GLKG}^{GLGK})(s_G)\in A_K(GLG)=A(GLGK)$. This element $u$ gives in turn the morphism
$$v:\theta_K\psi_K(G)=\rho_K(G)\to\rho_L(G)$$
equal to $u\in A(GLGK)$, but viewed as a morphism in $\cp_A$ from $GK$ to $GL$.
\mpn
(3) This is clear, by adjunction.
\end{proof}

\begin{prop}
\label{zisgbf}
The center of $A$ is a {$\CD$-}biset functor.
\end{prop}
\begin{proof} First $ZA(L)$ is obviously an $R$-module, for any $L\in\CD$. Let $K\in \CD$ and $t\in ZA(K)$, i.e. $t$ is a natural transformation $\Id\to\rho_K$ of endofunctors of the category~$\cp_A$. If $L\in\CD$ and $u\in RB(LK)$, let $u_A=\upsilon_{LK}(u)\in A(LK)$ be the image of $u$ under the unique morphism of Green functor $\upsilon:RB\to A$. Since $u_A\in CA(LK)$, by Corollary~\ref{bisets commute}, we can compose $t$ with the natural transformation $\rho_{u_A}:\rho_K\to \rho_L$ from Proposition~\ref{natural xi}, to get a natural transformation $\rho_{u_A}\circ t:\Id\to \rho_L$, i.e. an element of $ZA(L)$. Hence we get a linear map
$$u\in RB(LK)\mapsto \Big(t\mapsto \rho_{u_A}\circ t\in\Hom_R\big(ZA(K),ZA(L)\big)\Big),$$
and Assertion~2 of Proposition~\ref{natural xi} shows that this endows $ZA$ with a structure of biset functor.
\end{proof}

We now build a product on $ZA$, to make it a Green biset functor. For $K,L\in \CD$, let $s\in ZA(K)$ and $t\in ZA(L)$. Since $s$ is  a natural transformation $\Id\to \rho_K$, we get, by adjunction, a natural transformation $s^o:\rho_K\to\Id$. By composition with $t:\Id\to \rho_L$, we obtain a natural transformation $t\circ s^o:\rho_K\to\rho_L$, which in turn, by adjunction again, gives a natural transformation $^o(t\circ s^o):\Id\to (\rho_L)_K\cong\rho_{LK}$, i.e. an element of $ZA(LK)$. So we set
\begin{equation}\label{prod ZA}\forall s\in ZA(K),\;\forall t\in ZA(L),\;\;t\times s={^o}(t\circ s^o)\in ZA(LK).
\end{equation}

Translating this in the terms of Proposition~\ref{ZA Cr} gives:

\begin{lema} \label{product Cr}Let $s\in ZA(K)$ and $t\in ZA(L)$ be defined respectively by families of elements $s_G\in A(GKG)$ and $t_G\in A(GLG)$, for $G\in \CD$. Then $t\times s$ is the element of $ZA(LK)$ defined by the family $(t\times s)_G=t_G\circ s_G\in A(GLKG)$, for $G\in \CD$.
\end{lema}
\begin{proof} \label{traduction}As the adjunction $s\mapsto s^o$ amounts to switching the last two components of $GKG$, the element $t\times s={^o}(t\circ s^o)$ is defined by the family
\begin{align*}
(t\times s)_G&=A(\Iso_{GLGK}^{GLKG})\big(t_G\circ A(\Iso_{GKG}^{GGK})(s_G)\big)\\
&=A(\Iso_{GLGK}^{GLKG})A(\Def_{GLGK}^{GL\sou{G}GK}\Res_{GL\sou{G}GK}^{GL\sou{GG}GK})\big(t_G\times A(\Iso_{GKG}^{GGK})(s_G)\big),\\
\end{align*}
where the notation $\Def_{GLGK}^{GL\sou{G}GK}$ means that we take deflation with respect to the underlined factor, and $\rule{0ex}{3ex}\Res_{GL\sou{G}GK}^{GL\sou{GG}GK}$ means that the underlined $G$ in subscript embeds diagonally in the underlined group $GG$ in superscript. It follows that
\begin{align*}
(t\times s)_G&=A(\Def_{GLKG}^{GL\sou{G}KG}\Iso_{GLGGK}^{GLGKG}\Res_{GL\sou{G}GK}^{GL\sou{GG}GK}\Iso_{GLGGKG}^{GLGGGK})(t_G\times s_G)\\
&=A(\Def_{GLKG}^{GL\sou{G}KG}\Res_{GL\sou{G}KG}^{GL\sou{GG}KG})(t_G\times s_G)\\
&=t_G\circ s_G\in A(GLKG).
\end{align*}
\end{proof}

\begin{nota} Let $A$ be a Green $\CD$-biset functor, and $G,H,K,L\in \CD$. For morphisms in $\cp_A$,   namely $\alpha :G\to H$ in $A(HG)$ and $\beta :K\to L$ in $A(LK)$, we denote by $\alpha\boxtimes\beta:GK\to HL$ the morphism defined by 
$$\alpha\boxtimes \beta=A(\Iso_{HGLK}^{HLGK})(\alpha\times\beta)\in A(HLGK).$$
\end{nota}
\begin{prop} \label{produits} Let $A$ be a Green $\CD$-biset functor, and $G,H,K,L\in \CD$. Let moreover $\alpha\in CA(HG)$ and $\beta\in CA(LK)$. Then for any $s\in ZA(G)$ and $t\in ZA(K)$, and for any $X\in\CD$
$$(\rho_\alpha\circ s)_X\circ (\rho_\beta\circ t)_X=\big(\rho_{\alpha\boxtimes\beta}\circ(s\times t)\big)_X.$$
\end{prop}
\begin{proof} The proof amounts to rather lengthy but straighforward calculations on bisets, similar to those we already did several times above, e.g. in the proof of Theorem~\ref{adjonctions-enonce}. We leave it as an exercice.\end{proof}
\begin{teo} Let $A$ be a Green $\CD$-biset functor. Then $ZA$, endowed with the product defined in~(\ref{prod ZA}), is a Green $\CD$-biset functor.
\end{teo}
\begin{proof} It is clear from Lemma~\ref{product Cr} and Proposition~\ref{ZA Cr} that the product on $ZA$ is associative. Moreover the identity transformation from the identity functor to $\rho_1=\Id_{\cp_A}$ is obviously an identity element for the product on $ZA$. This product is also $R$-bilinear by construction. Finally, the equality $ZA(U)(s)\times ZA(V)(t)=ZA(U\boxtimes V)(s\times t)$ for bisets $U$ and $V$ is a special case of Proposition~\ref{produits}.
\end{proof}
\subsection{Relations between the commutant and the center}

\begin{prop} \label{CA ZA A}Let $A$ be a Green $\CD$-biset functor. 
\begin{enumerate}
\item The maps sending $\alpha\in CA(L)$ to $\rho_\alpha\in ZA(L)$, for $L\in \CD$, define a morphism of Green biset functors $\iota_A:CA\to ZA$.
\item The maps sending $t\in Cr_A(L)\cong ZA(L)$ to $t_1\in A(L)$, for $L\in \CD$, define a morphism of Green biset functors $\pi_A:ZA\to A$. The image of this morphism in the component 1 lies in $Z(A(1))$, hence there is a morphism of rings $\pi_{A,\,1}:ZA(1)\rightarrow Z(A(1))$.
\item The composition 
$$\xymatrix{
CA\;\ar@{^(->}[r]^-{\iota_A}&ZA\ar[r]^-{\pi_A}&A
}$$
is equal to the inclusion $\xymatrix{CA\;\ar@{^(->}[r]&A}$. In particular $\iota_A$ is injective.
\end{enumerate}
\end{prop}
\begin{proof} For Assertion 1, let $\alpha\in CA(K)$, for $K\in\CD$. Then the element $\rho_\alpha$ of $ZA(K)$ corresponds to the family of elements $\rho_{\alpha,G}\in A(GKG)$, for $G\in\CD$, defined by
$$\rho_{\alpha,G}=A(\Ind_{KG}^{GKG}\Inf_{K}^{KG})(\alpha).$$
Similarly, if $L\in\CD$ and $\beta\in CA(L)$, the element $\rho_\beta$ of $ZA(L)$ corresponds to the family $\rho_{\beta,G}=A(\Ind_{LG}^{GLG}\Inf_{L}^{LG})(\beta)$. By Lemma~\ref{product Cr}, the product $q=\rho_\alpha\times \rho_\beta$ in $ZA(KL)$ corresponds to the family
\begin{align*}
q_G&= \rho_{\alpha,G}\circ \rho_{\beta,G}\\
&= A(\Ind_{KG}^{GKG}\Inf_{K}^{KG})(\alpha)\circ A(\Ind_{LG}^{GLG}\Inf_{L}^{LG})(\beta)\\
&= A(\Def_{GKLG}^{GK\sou{G}LG}\Res_{GK\sou{G}LG}^{GK\sou{GG}LG})\big(A(\Ind_{KG}^{GKG}\Inf_{K}^{KG})(\alpha)\times A(\Ind_{LG}^{GLG}\Inf_{L}^{LG})(\beta)\big)\\
&=A(\Def_{GKLG}^{GK\sou{G}LG}\Res_{GK\sou{G}LG}^{GK\sou{GG}LG}\Ind_{KGLG}^{GKGGLG}\Inf_{KL}^{KGLG})(\alpha\times\beta).
\end{align*}
Standard relations in the composition of bisets then show that
$$q_G=A(\Ind_{KLG}^{GKLG}\Inf_{KL}^{KLG})(\alpha\times\beta),$$
and it follows that $q=\rho_{\alpha\times\beta}$. In other words $\iota_A(\alpha\times\beta)=\iota_A(\alpha)\times\iota_A(\beta)$. Moreover, the identity element $\varepsilon_A\in CA(1)$ is mapped by $\iota_A$ to the element of $ZA(1)$ defined by the family of elements $A(\Ind_{G}^{GG}\Inf_1^G)(\varepsilon_A)$, for $G\in\CD$, that is the identity element of $ZA$. So $\iota_A$ is a morphism of Green $\CD$-biset functors.\mp
The first part of Assertion 2 is a consequence of Lemma~\ref{product Cr}. Indeed, if $K,L\in \CD$, if $s\in ZA(K)$ corresponds to the family $s_G\in Cr_A(K)$, and if $t\in ZA(L)$ corresponds to the family $\beta_G\in Cr_A(L)$, for $G\in \CD$, then the product $u=s\times t$ is the element of $ZA(KL)$ corresponding to the family $u_G=s_G\circ t_G$. In particular, for $G=1$, we have
$$u_1=s_1\circ t_1=s_1\times t_1.$$
This shows that the maps sending $t\in ZA(L)$ to $t_1\in A(L)$, for $L\in\CD$, is a morphism of Green $\CD$-biset functors $\pi:ZA\to A$.\par
Since composition $\circ:A(1)\times A(1)\rightarrow A(1)$ coincides with the product  of~$A(1)$ as a ring, the commutativity property defining the series of $Cr_A(1)$ shows that $\pi_{A,\, 1}$ has image in $Z(A(1))$. This completes the proof of Assertion~2.\mp
For Assertion 3, we start with an element $\alpha\in CA(L)$, for $L\in\CD$. It is sent by $\iota_A$ to the element $t\in ZA(L)$ corresponding to the family $t_G=A(\Ind_{LG}^{GLG}\Inf_{L}^{LG})(\alpha)$, for $G\in \CD$, in $Cr_A(L)$. In particular $t_1=A(\Ind_L^L\Inf_L^L)(\alpha)=\alpha$, so $\pi_A\circ \iota_A$ is equal to the inclusion $CA\hookrightarrow A$.
\end{proof}

The morphism $\iota_A$ of the previous proposition allows us to give a $CA$-module structure to $ZA$. With this structure, (the image under $\iota_A$ of) $CA$ is a $CA$-submodule of~$ZA$. In the particular case where $A$ is commutative, the previous proposition tells us more.

\begin{coro}
If $A$ is a commutative Green $\CD$-biset functor, then $A$ is isomorphic to a direct summand of $ZA$ in the category $A$-Mod.
\end{coro}

\begin{proof} This follows from the fact that $\iota_A$ and $\pi_A$ are morphisms of Green $\CD$-biset functors, so in particular morphisms of $A$-modules. Moreover the composition $\pi_A\circ\iota_A$ is equal to the identity when $A$ is commutative.\end{proof}

\begin{prop} Let $A$ be a Green $\CD$-biset functor. Let $\End_R(\cp_A)$ be the category of $R$-linear endofunctors of $\cp_A$. 
\begin{enumerate}
\item The assignment
$$\left\{\begin{array}{rcl}K\in\CD&\mapsto& \rho_K\in\End_R(\cp_A)\\
t\in ZA(LK)&\mapsto& t^o\in\Hom_{\End_R(\cp_A)}(\rho_K,\rho_L)\end{array}\right.$$
is a fully faithful $R$-linear functor $\rho_{ZA}$ from $\cp_{ZA}$ to $\End_R(\cp_A)$.
\item The following assignment $\mu_A$
$$\left\{\begin{array}{rcl}K\in\CD&\mapsto& K\in\CD\\
\alpha\in CA(LK)&\mapsto& {^o\rho}_\alpha\in ZA(LK)\end{array}\right.$$
is equal to the functor $\cp_{\iota_A}$ from $\cp_{CA}$ to $\cp_{ZA}$, induced by $\iota_A:CA\to ZA$. In particular $\mu_A$ is faithful, and such that
$$\rho_{ZA}\circ \mu_A=\rho_{CA}.$$
\item The following assignment $\nu_A$ 
$$\left\{\begin{array}{rcl}K\in\CD&\mapsto& K\in\CD\\
s\in ZA(LK)&\mapsto& s_1\in A(LK)\end{array}\right.$$
is equal to the functor $\cp_{\pi_A}$ from $\cp_{ZA}$ to $\cp_{A}$ induced by the morphism of Green biset functors $\pi_A:ZA\to A$. The composition $\nu_A\circ \mu_A$ is equal to the inclusion functor $\cp_{CA}\to \cp_A$. 
\end{enumerate}
\end{prop}
\begin{proof} All the assertions are straightforward consequences of Proposition~\ref{CA ZA A}.
\end{proof}
To conclude this section, we now show that the isomorphism $CA_L\cong (CA)_L$ of Proposition~\ref{commsft} only extends to an injection $ZA_L\hookrightarrow (ZA)_L$. We first prove a lemma.
\begin{lema} \label{psipsi}Let $A$ be a Green $\CD$-biset functor. For $L\in\CD$, let $\psi_L^A:\cp_A\to \cp_{A_L}$ be the functor $\psi_L$ of Theorem~\ref{adjonctions-enonce}. If $K\in\CD$, let $\psi_K^{A_L}:\cp_{A_L}\to\cp_{(A_L)_K}$ be the similar functor built from $A_L$ and $K$. Then the diagram 
$$\xymatrix{
\cp_A\ar[r]^-{\psi_L^A}\ar[drr]_-{\psi_{KL}^A}&\cp_{A_L}\ar[r]^-{\psi_K^{A_L}}&\cp_{(A_L)_K}\ar[d]_-{e_{K,L}}^-\cong\\
&&\cp_{A_{KL}}
}
$$
of categories and functors, is commutative, where $e_{K,L}$ is the natural equivalence of categories $\cp_{(A_L)_K}\to \cp_{A_{KL}}$ provided by the canonical isomorphism of Green $\CD$-biset functors $(A_L)_K\cong A_{KL}$.
\end{lema}
\begin{proof} Indeed, all the functors involved are the identity on objects. And for a morphism $\alpha:G\to H$ in $\cp_A$, i.e. an element $\alpha$ of $A(HG)$, we have 
\begin{align*}
\psi_K^{A_L}\psi_L^A(\alpha)&=\psi_K^{A_L}A(\Inf_{HG}^{HGL})(\alpha)=A_L(\Inf_{HG}^{HGK})A(\Inf_{HG}^{HGL})(\alpha)\\
&=A(\Inf_{HGL}^{HGKL})A(\Inf_{HG}^{HGL})(\alpha)\\
&=A(\Inf_{HG}^{HGKL})(\alpha)=\psi_{KL}^A(\alpha).
\end{align*}
\end{proof}

\begin{prop}
\label{centersft}
Let $A$ be a Green biset functor and $L\in\CD$. Then there is an injective morphism of Green $\CD$-biset functors from $ZA_L$ to $(ZA)_L$.
\end{prop}

\begin{proof}
Let $K,L\in\CD$, and $t\in ZA_L(K)$, i.e. a natural transformation
$$t:\Id_{\cp_{A_L}}\to \rho_K^{A_L}$$
from the identity functor of $\cp_{A_L}$ to the functor $\rho_K^{A_L}=\theta_K^{A_L}\psi_K^{A_L}$, where $\theta_K^{A_L}$ is the functor $\cp_{(A_L)_K}\to \cp_{A_{L}}$ of Theorem~\ref{adjonctions-enonce} built from $A_L$ and $K$.  By precomposition of this natural transformation with the functor $\psi_L^A$, we get a natural transformation
$$\psi_L^A\to \theta_K^{A_L}\psi_K^{A_L}\psi_L^A,$$
which by adjunction, gives a natural transformation 
$$\Id_{\cp_A}\to \theta_L^A\theta_K^{A_L}\psi_K^{A_L}\psi_L^A.$$
By Lemma~\ref{psipsi}, the functor $\psi_K^{A_L}\psi_L^A$ is isomorphic to $\psi_{KL}^A$. By Theorem~\ref{adjonctions-enonce}, the functor $\theta_K^{A_L}$ is left and right adjoint to the functor $\psi_K^{A_L}$, and $\theta_L^A$ is left and right adjoint to $\psi_L^A$. It follows that the functor $\theta_L^A\theta_K^{A_L}$ is isomorphic to the adjoint $\theta_{KL}^A$ of $\psi_{KL}^A$. Hence we have a natural transformation
$$T:\Id_{\cp_A}\to \theta_{KL}^A\psi_{KL}^A=\rho_{KL}^A,$$
that is an element of $ZA(KL)=(ZA)_L(K)$.\par
So we have a map $j_{L,K}:t\in ZA_L(K)\mapsto T\in(ZA)_L(K)$, which is obviously $R$-linear. Lengthy but straightforward calculations show that the family of these maps, for $K\in\CD$, form a morphism of Green biset functors from $ZA_L$ to $(ZA)_L$. 
\end{proof}

\section{Application: some equivalences of categories}\label{equiv}

\subsection{General setting}\label{genset}

We  begin by recalling some well known folklore facts on the decomposition of a category $\mathcal{F}_\cp$ of functors from a small $R$-linear category $\cp$ to $\AMod{R}$, using an orthogonal decomposition of the identity in the center $Z\cp$ of $\cp$. \par
Since $\cp$ is $R$-linear, its center $Z\cp$ is a commutative $R$-algebra. Suppose we have a family $(\gamma_i)_{i\in I}$ of elements of $Z\cp$ indexed by a set $I$, with the following properties:
\begin{enumerate}
\item For $i,j\in I$, the product $\gamma_i\gamma_j$ is equal to 0 if $i\neq j$, and to $\gamma_i$ if $i=j$.
\item For any object $G$ of $\cp$, there is only a finite number of elements $i\in I$ such that $\gamma_{i,G}\neq 0$. Then, for each object $G\in \cp$, we can consider the (finite) sum $\sum_{i\in I}\limits \gamma_{i,G}$, which is a well defined endomorphism of $G$. We assume that this endomorphism is the identity of $G$, for any $G\in\cp$.
\end{enumerate}
If $F$ is an $R$-linear functor from $\cp$ to $\AMod{R}$, and $i\in I$, we denote by  $F\gamma_i$ the functor that in an object $G$ of $\cp$ is defined as the image of $F(\gamma_{i,G})$, that is 
$$(F\gamma_i)(G)=\Im \big(F(\gamma_{i,G}):F(G)\to F(G)\big),$$
 which is an $R$-submodule of $F(G)$. For a morphism $\alpha:G\to H$, we denote by $(F\gamma_i)(\alpha)$ the restriction of $F(\alpha)$ to $(F\gamma_i)(G)$. The image of $(F\gamma_i)(\alpha)$ is contained in $F\gamma_i(H)$, because the square
$$\xymatrix{
G\ar[r]^-{\gamma_{i,G}}\ar[d]_-\alpha&G\ar[d]^-\alpha\\
H\ar[r]_-{\gamma_i,H}&H
}
$$
is commutative in $\cp$, hence also its image by $F$.\par
It is easy to check that $F\gamma_i$ is an $R$-linear functor from $\cp$ to $\AMod{R}$, which is a subfunctor of $F$. Moreover, the assigment $F\mapsto F\gamma_i$ is an endofunctor $\Gamma_i$ of the category $\mathcal{F}_\cp$. The image of this functor consists of those functors $F\in\mathcal{F}_\cp$ such that the subfunctor $F\gamma_i$ is equal to $F$. Let $\mathcal{F}_\cp\gamma_i$ be the full subcategory of $\mathcal{F}_\cp$ consisting of such functors. It is an abelian subcategory of $\mathcal{F}_\cp$.\par
For each $G\in\cp$, the direct sum $\mathop{\oplus}_{i\in I}\limits F\gamma_i(G)$ is actually finite, and our assumptions ensure that is is equal fo $F(G)$. This shows that the functor sending $F\in\mathcal{F}_\cp$ to the family of functors $F\gamma_i$ is an equivalence between $\mathcal{F}_\cp$ and the product of the categories~$\mathcal{F}_\cp\gamma_i$. 

A particular case of the previous situation is when
the identity element $\varepsilon\in A(1)$ of a Green biset functor $A$ has a decomposition in  orthogonal idempotents $\varepsilon=\sum_{i=1}^ne_i$ in the ring $CA(1)$.  Each $e_i$ induces a natural transformation $E^i:Id\rightarrow Id_1$, defined {at} an $A$-module $M$ {and a group $H\in\CD$} as
\begin{displaymath}
E^i_{M,\, H}:M(H)\rightarrow M_1(H)\quad m\mapsto M(\Iso^{H\times 1}_{1\times H})(e_i\times m).
\end{displaymath}
For simplicity, we will think of this natural transformation as sending $m$ simply to $e_i\times m$, and we will denote by $e_iM$ the $A$-submodule of $M$ given by the image of $E^i_M$.

Since the morphism from $CA(1)$ to $ZA(1)$ is a ring homomorphism, we have that the natural transformations $E^i$ satisfy $E^i\circ E^i=E^i$, $E^i\circ E^j=0$ if $i\neq j$ and that the identity natural transformation, $\mathbf{1}$, is equal to $\sum_{i=1}^nE^i$.  By Proposition \ref{ZAZA'}, we have then the hypothesis assumed at the beginning of the section and so we obtain the equivalence of categories mentioned above. In this case we can give a more precise description of this equivalence.

\begin{lema}
The $A$-module $e_iA$ is a Green {$\CD$-}biset functor, and for every $A$-module~$M$, the functor $e_iM$ is an $e_iA$-module. Furthermore $A\cong \bigoplus_{i=1}^ne_iA$ as Green {$\CD$-}biset functors.
\end{lema}
\begin{proof}
As we have said, $e_iA$ is an $A$-module, in particular it is a biset functor. 
We claim that it is a Green biset functor with the product
\begin{displaymath}
e_iA(G)\times e_iA(K)\rightarrow e_iA(G\times K)\quad (e_i\times a)\times (e_i\times b)=e_i\times a\times b.
\end{displaymath}
Observe that since all the $\times$ represent the product of $A$, then $(e_i\times a)\times (e_i\times b)$ is isomorphic to $a\times e_i\times e_i\times b$, because $e_i\in CA(1)$. But the product $\times$ coincides with the ring product in $A(1)$, hence this element is isomorphic to $a\times e_i\times b$ and then to $e_i\times a\times b$. 
This implies immediately that the product is associative, the identity element in $e_iA(1)$ is of course $e_i\times \varepsilon$. Next, notice that since $E^i_A$ is a morphism of {$A$-modules}, if {$L,G\in\CD$} and $X$ is an $(L,\, G)$-biset, then $A(X)(e_i\times a)\cong e_i\times A(X)(a)$ for all $a\in A(G)$. With this, one can easily show the functoriality of the product.

Similar arguments show  that $e_iM$ is an $e_iA$-module with the product
\begin{displaymath}
e_iA(G)\times e_iM(K)\rightarrow e_iM(G\times K)\quad (e_i\times a)\times (e_i\times m)=e_i\times a\times m.
\end{displaymath}

For the final statement, first it is an easy exercise to verify that given Green biset functors $A_1,\ldots ,\, A_r$, then their direct sum $\bigoplus_{i=1}^rA_i$ in the category of biset functors is again a Green biset functor, with the product given component-wise. With this, it is straightforward to see that the isomorphism of biset functors $A\cong\bigoplus_{i=1}^ne_iA$ is an isomorphism of Green biset functors.
\end{proof}
 All these observations give us the following result.

\begin{teo}\label{decomposition}
Let $A$ be a Green {$\CD$-}biset functor as above. Then the category $\AMod{A}$ is equivalent to the product category
\begin{displaymath}
\prod_{i=1}^n\AMod{e_iA}.
\end{displaymath}
Moreover, for each indecomposable $A$-module $M$, there exists only one $e_i$ such that $e_iM\neq 0$, and hence $e_iM\cong M$.
\end{teo}

When considering the shifted functor $A_H$, if we have an idempotent $e\in CA_H(1)$ as before, then the evaluation of $eA_H$ at a group $G$ can be seen as follows. Since $eA_H(G)=e\times^d A_H(G)$, then
for $a\in A_H(G)$ it is easy to see that
\begin{displaymath}
e\times^d a=A(\Res^{1\times H\times G\times H}_{G\times \Delta(H)})(e\times a)=A(\Inf_H^{G\times H})(e)\cdot a, 
\end{displaymath}
where the product $\cdot$ indicates the ring structure in $A(G\times H)$. The last equality follows from Lemma \ref{defeq} and the properties of restriction and inflation. So, the evaluation of $eA_H$ at a given group depends on how inflation of $A$ acts on the idempotents of $CA(H)$.

\subsection{Some examples}\label{some examples}

\subsubsection{$p$-biset functors}
Let $p$ be a prime, and $RB_p$ denote the restriction to finite $p$-groups of the Burnside functor $RB$ of Example 5. When $p$ is invertible in the ring $R$, a family of orthogonal idempotents in the center of the Green biset functor $RB_p$ of Example~\ref{ejemplos} has been introduced in~\cite{atoric}. These idempotents $\widehat{b}_L$ are indexed by {\em atoric} $p$-groups $L$ up to isomorphism, i.e. finite $p$-groups which cannot be decomposed as a direct product $Q\times C_p$ of a finite $p$-group $Q$ and a group $C_p$ of order $p$. \par
More precisely, for each such atoric $p$-group $L$ and each finite $p$-group $P$, a specific idempotent $b_L^P$ of $RB_p(P,P)$ is introduced (cf. \cite{atoric}, Theorem 7.4), with the property that 
$$a\circ b_L^P=b_L^Q\circ a$$
for any finite $p$-groups $P$ and $Q$, and any $a\in RB(Q,P)$. In other words, the family $b_L=(b_L^P)_P$ is an element of the center of the biset category $R\mathcal{C}_p$ of finite $p$-groups. The elements $\widehat{b}_L$ of the center of the category of $p$-biset functors over $R$ - i.e. the category $\AMod{RB_p}$ - are deduced from the elements $b_L$ in~\cite{atoric}, Corollary~7.5.\par
Let $[\mathcal{A}t_p]$ denote a set of representatives of isomorphism classes of atoric $p$-groups. The idempotents $b_L^P$ have the following additional properties: 
\begin{enumerate}
\item  If $L$ and $L'$ are isomorphic atoric $p$-groups, then $b_L^P=b_{L'}^P$
\item  If $L$ and $L'$ are non isomorphic atoric $p$-groups, then $b_L^Pb_{L'}^P=0$.
\item  For a given finite $p$-group $P$, there are only a finite number of atoric $p$-groups $L$, up to isomorphism, such that $b_L^P\neq 0$. 
\item  The sum $\sum_{L\in[\mathcal{A}t_p]}\limits b_L^P$, which is a finite sum by the previous property, is equal to the identity element of $RB(P,P)$.
\end{enumerate}
It follows that one can consider the sum $\sum_{L\in[\mathcal{A}t_p]}\limits \widehat{b}_L$ in $Z(RB_p)(1)$, and that this sum is equal to the identity element of $Z(RB_p)(1)$. So we obtain a {\em locally finite} decomposition of the identity element of $Z(RB_p)(1)$ as a sum of orthogonal idempotents, which allows for a splitting of the category of $p$-biset functors over $R$ as a direct product of abelian subcategories (cf.~\cite{atoric}, Corollary 7.5). As a consequence, for each indecomposable $p$-biset functor $F$ over $R$, there is an atoric $p$-group $L$, unique up to isomorphism, such that $\widehat{b}_L$ acts as the identity of $F$ (or equivalently, does not act by zero on $F$). This group $L$ is called the {\em vertex} of $F$ (cf.~\cite{atoric}, Definition 9.2).
\begin{rem} This example shows in particular that $ZA$ can be much bigger than $CA$: indeed for $A=RB_p$, when $R$ is a field of characteristic different from $p$, we see that $ZA(1)$ is an infinite dimensional $R$-vector space, whereas $CA(1)\cong R$ is one dimensional.
\end{rem}
\subsubsection{Shifted representation functors}

Now we apply the results of Section \ref{genset} to some shifted classical representation functors, with coefficients in a field $\mathbb{F}$ of characteristic 0. In each case we will begin with a commutative Green biset functor $C$ such that for each group $H$, the $\mathbb{F}$-algebra $C(H)$ is split semisimple. In particular, taking $A=C_H$, in $A(1)=C(H)$ we will have a family of orthogonal idempotents $\{e_i^H\}_{i=1}^{n_H}$ such that $\varepsilon=\sum_{i=1}^{n_H}\limits e_i^H$. As we said in Section \ref{genset}, the evaluation $e_i^HA(G)$ is given in the following way
\begin{displaymath}
e_i^H\times^da=A(\Inf_1^{G})(e_i^H)\cdot a=C(\Inf_{H}^{G\times H})(e_i^H)\cdot a
\end{displaymath}
for $a\in A(G)$. Now, since inflation is a ring homomorphism, $A(\Inf_1^{G})(e_i^H)$ is equal to $\sum_{j\in J}\limits e_j^{G\times H}$ for some $J\subseteq \{1,\ldots, n_{G\times H}\}$ depending on $e_i^H$ and $G$. On the other hand, we also have $a=\sum_{i=1}^{n_{G\times H}}\limits \alpha_i(a)e_i^{G\times H}$, for some $\alpha_i(a)\in \mathbb{F}$. This implies that the idempotents appearing in the evaluation $e_i^{H}A(G)$ depend only on the set $\{e_j^{G\times H}\}_{j\in J}$. 

 \subsubsection*{Shifted Burnside functors.}

We consider the Burnside functor $\mathbb{F}B$ over $\mathbb{F}$. We fix a finite group $H$, and consider the shifted functor $A=\mathbb{F}B_H$. Then the algebra $A(1)$ is isomorphic to $\mathbb{F}B(H)$, hence it is split semisimple. Its primitive idempotents $e_K^H$ are indexed by subgroups $K$ of $H$, up to conjugation, and explicitly given (see. \cite{gluck}, \cite{yoshidaidemp}) by
$$e_K^H=\frac{1}{|N_H(K)|}\sum_{L\leq K}|K|\mu(L,K)\,[H/L],$$
where $\mu$ is the M\"obius function of the poset of subgroups of $H$ and $[H/L]\in B(H)$ is the class of the transitive $H$-set $H/L$. \par
By Theorem~\ref{decomposition}, we get a decomposition of the category $\AMod{A}$ as a product $\prod_{K\in[s_H]}\limits \AMod{e_K^H A}$, where $[s_H]$ is a set of representatives of conjugacy classes of subgroups of $H$. From the action of inflation on the primitive idempotents of Burnside rings (see~\cite{biset} Theorem 5.2.4), it is easy to see that for $K\leq H$,  the value $e_K^H A(G)$ of the Green functor $e_K^H A$ at a finite group $G$ is equal to the set of linear combinations of idempotents $e_L^{G\times H}$ of $\mathbb{F}B(G\times H)$ indexed by subgroups $L$ of $(G\times H)$ for which the second projection $p_2(L)$ is conjugate to $K$ in $H$. Also, for each indecomposable $A$-module $M$, there exists a unique $K\leq H$, up to conjugation, such that $e_K^HM\neq 0$, and then $e_K^HM=M$.

 \subsubsection*{Shifted functors of linear representations.}

Next we consider the functor $\mathbb{F}R_\mathbb{K}$ of linear representations over $\mathbb{K}$, a field of characteristic 0. As before, we fix a finite group $H$ and consider the shifted functor $A=(\mathbb{F}R_\mathbb{K})_H$. This is a commutative Green biset functor, and $A(1)$ is isomorphic to the split semisimple $\mathbb{F}$-algebra $\mathbb{F}R_\mathbb{K}(H)$. If $|H|=n$, it is shown in Section 3.3.1 of \cite{thbenja} (and in a slightly different way in \cite{artbenja}) that $\mathbb{F}R_{\mathbb{K}}(H)$ has a complete family of orthogonal primitive idempotents $e_D^H$ indexed by the $E$-conjugacy classes of $H$, where $E$ is certain subgroup of $(\mathbb{Z}/n\mathbb{Z})^\times$. By $E$-conjugacy we mean that two elements $x,\, y\in H$ are $E$-conjugated if there exist $[i]\in E$ such that $x=_Hy^i$. This defines an equivalence relation on $H$ and the set of $E$-conjugacy classes is denoted by $Cl_E(H)$. The group $E$ is built in the following way: First we fix an algebraically closed field $\mathbb{L}$, which is an extension of $\mathbb{F}$ and $\mathbb{K}$, and then we take the intersection $\mathbb{E}=\mathbb{F}\cap\mathbb{K}$ in $\mathbb{L}$. By adding an $n$-th primitive root of unity, $\omega$, to $\mathbb{E}$, we obtain $E$ as the group isomorphic to $Gal(\mathbb{E}[\omega]/\mathbb{E})$ in $(\mathbb{Z}/n\mathbb{Z})^\times$.  Observe that, as a group, $E$ depends only on $\mathbb{F}$, $\mathbb{K}$ and $n$, and not on the choice of $\mathbb{L}$.  Then, by Theorem~\ref{decomposition}, we get a decomposition of the category $\AMod{A}$ as a product $\prod_{\substack{D\in Cl_E(H)}}\limits \AMod{e^H_D A}$.  Also, for each indecomposable $A$-module $M$, there exists a unique $E$-conjugacy class $D$ of $H$ such that $e_D^HM\neq 0$ and so $e_D^HM=M$. On the other hand, in Corollary 3.3.14 of \cite{thbenja} it is shown that $e^H_D A$ is a simple $A$-module and  hence that $A$ is a semisimple $A$-module, since $A=\sum_{D}\limits e^H_D A$. 

Finally, using  Lemma 3.3.10 in \cite{thbenja}, we see that the idempotents $e_C^{G\times H}$, for $C$ an $E$-class of $G\times H$, appearing in the evaluation $A(\binf_1^G)(e_D^H)$ are those for which $\pi_H(C)$, the projection of $C$ on $H$, is equal to $D$.

 \subsubsection*{Shifted $p$-permutation functors.}

Let $k$ be an algebraically closed field of positive characteristic $p$. In this case we assume also that $\mathbb{F}$ contains all the $p'$-roots of unity, and consider the functor $\mathbb{F}pp_k$. Then $\mathbb{F}pp_k$ is a commutative Green biset functor, and the category $\AMod{\mathbb{F}pp_k}$ has been considered in particular in~\cite{ducellier} (when $\mathbb{F}$ is algebraically closed).

We fix a finite group $H$, and consider the shifted functor $A=(\mathbb{F}pp_k)_H$. Then the algebra $A(1)$ is isomorphic to the algebra $\mathbb{F}pp_k(H)$. This algebra is split semisimple, and its primitive idempotents $F_{Q,s}^H$ have been determined in~\cite{both5}: they are indexed by (conjugacy classes of) pairs $(Q,s)$ consisting of a $p$-subgroup $Q$ of~$H$, and a $p'$-element $s$ of $N_H(Q)/Q$. We denote by $\mathcal{Q}_{H,p}$ the set of such pairs, and by $[\mathcal{Q}_{H,p}]$ a set of representatives of orbits of $H$ for its action on $\mathcal{Q}_{H,p}$ by conjugation. \par
If $(Q,s)\in \mathcal{Q}_{H,p}$ and $u\in \mathbb{F}pp_k(H)$, then $F_{Q,s}^Hu=\tau_{Q,s}^H(u)F_{Q,s}^H$, where $\tau_{Q,s}^H(u)\in\mathbb{F}$. The maps $u\mapsto \tau_{Q,s}^H(u)$, for $(Q,s)\in[\mathcal{Q}_{H,p}]$ are the distinct algebra homomorphisms (the species) from $\mathbb{F}pp_k(H)$ to $\mathbb{F}$ (see e.g.~\cite{both5} Proposition 2.18). Moreover, the map $\tau_{Q,s}^H$ is determined by the fact that for any $p$-permutation $kH$-module $M$, the scalar $\tau_{Q,s}^H(M)$ is equal to the value at $s$ of the Brauer character of the Brauer quotient $M[Q]$ of $M$ at $Q$. \par
It follows that if $N\trianglelefteq H$, and $v\in\mathbb{F}pp_k(H/N)$, then $\tau_{Q,s}^H(\Inf_{H/N}^Hv)=\tau_{\sur{Q},\sur{s}}^{H/N}(v)$, where $\sur{Q}=QN/N$, and $\sur{s}\in N_{H/N}(\sur{Q})/\sur{Q}$ is the projection of $s$ to $H/N$. As a consequence, if $(R,t)\in\mathcal{Q}_{H/N,p}$, then $\Inf_{H/N}^H(F_{R,t}^{H/N})$ is equal to the sum of the idempotents $F_{Q,s}^H$ for those elements $(Q,s)\in[\mathcal{Q}_{H,p}]$ for which $(\sur{Q},\sur{s})$ is conjugate to $(R,t)$ in $H/N$.\par
Now by Theorem~\ref{decomposition}, we get a decomposition of the category $\AMod{A}$ as a product $\prod_{(Q,s)\in [\mathcal{Q}_{H,p}]}\limits \AMod{F_{Q,s}^HA}$. Let $G$ be a finite group. It follows from the previous discussion on inflation that the evaluation $F_{Q,s}^HA(G)$ of $A$ at $G$ is equal to the set of linear combinations of primitive idempotents $F_{L,t}^{G\times H}$, for $(L,t)\in\mathcal{Q}_{G\times H,p}$, such that the pair $\big(p_2(L),p_2(t)\big)$ is conjugate to $(Q,s)$ in $H$, where $p_2:G\times H\to H$ is the second projection. Also, for each indecomposable $A$-module~$M$, there exists a unique $(Q,s)\in[\mathcal{Q}_{H,p}]$ such that $F_{Q,s}^HM\neq 0$, and then $F_{Q,s}^HM=M$.

{
\centerline{\rule{5ex}{.1ex}}
\vspace{1cm}
\begin{flushleft}
Serge Bouc, CNRS-LAMFA, Universit\'e de Picardie, 33 rue St Leu, 80039, Amiens, France.\\
{\tt serge.bouc@u-picardie.fr}\vspace{1ex}\\
Nadia Romero, DEMAT, UGTO, Jalisco s/n, Mineral de Valenciana, 36240, Guanajuato, Gto., Mexico.\\
{\tt nadia.romero@ugto.mx}
\end{flushleft}
}

\end{document}